\documentclass[12pt]{article}
\usepackage[utf8]{inputenc}
\usepackage[english]{babel}
\usepackage{amsmath}
\usepackage{amsthm}
\usepackage{amsfonts}
\usepackage{amssymb}
\usepackage{array}
\usepackage{tabularx}
\usepackage{graphicx}
\usepackage{tikz}
\usepackage{a4wide}

\newtheorem{thm}{Theorem}[section]

\newtheorem{lem}[thm]{Lemma}
\newtheorem{prop}[thm]{Proposition}
\theoremstyle{definition}

\theoremstyle{remark}

\title{Classifying edge-biregular maps \\ of negative prime Euler characteristic}

\author{Olivia Jeans${ }^1$ and Jozef \v{S}ir\'{a}\v{n}${ }^{1,2}$ \\  \ \\
\small{${ }^1$The Open University, Milton Keynes, UK} \\ \small{${ }^2$Slovak University of Technology, Bratislava, Slovakia}}

\date{}
\begin{document}

\maketitle

\begin{abstract}
\noindent An edge-biregular map arises as a smooth normal quotient of a unique index-two subgroup of a full triangle group acting with two edge-orbits. We give a classification of all finite edge-biregular maps on surfaces of negative prime Euler characteristic.
\end{abstract}

\section{Introduction}\label{sec:intro}

A {\em map} is a 2-cell embedding of a connected graph on a surface (which may be orientable or not but without boundary components). As such it is possible to form the barycentric subdivision of a map, and the resulting regions are the {\em flags} of the map. An {\em automorphism} of a map is a structure-preserving bijection of the set of flags onto itself, and the set of all automorphisms of a map form its {\em automorphism group}.
\smallskip

Informally, how much symmetry a map displays may be measured by the size of its automorphism group. If a map automorphism fixes a flag, then, by connectivity, it must fix all flags and hence be the identity; this means that the action of the automorphism group of a map is semi-regular on flags. A largest `level of symmetry' exhibited by a map therefore arises when the automorphism group is regular, that is semi-regular and transitive, on the set of flags. Such maps are {\em fully regular}; they are, in some sense, the most symmetric maps \cite{Sir-BCC}. Elements of the automorphism group of a fully regular map can be identified with flags in a one-to-one manner by fixing a flag and assigning to every other flag the unique automorphism taking the fixed flag to the other one.
\smallskip

An obvious consequence of full regularity of a map is that all its vertices have the same valency and all the faces have boundary walks of the same length; we speak about a map of {\em type} $(k,\ell)$ in the case of valency $k$ and face length $\ell$. As it turns out, every fully regular map of type $(k,\ell)$ is a smooth normal quotient of a {\em $(k,\ell)$-tessellation} of a simply connected surface by $\ell$-gons, $k$ of which meet at every vertex. Equivalently, in a group-theoretic language, the automorphism group of a fully regular map of type $(k,\ell)$ is a torsion-free normal quotient of the automorphism group of a $(k,\ell)$-tessellation; the latter is known as the full $(2,k,\ell)$-triangle group.
\smallskip

There are other classes of `highly symmetric' maps which are not necessarily fully regular but which have automorphism groups that can still be described as `large'. Perhaps the most natural candidate for a map to qualify for such a distinction would be when a group of automorphisms of the map has two (necessarily equal-sized) orbits and is regular on each. Following the example of existing graph-theoretic language \cite{DM, Mar, Zho}, we use the term {\em biregular} to describe such maps. As in the case of fully regular maps, automorphism groups of biregular maps of type $(k,\ell)$ also arise as torsion-free normal quotients, but this time of index-two subgroups of the full $(2,k,\ell)$-triangle groups. As we shall see, a full $(2,k,\ell)$-triangle group may contain up to $7$ subgroups of index two (depending on the parity of $k$ and $\ell$). This potentially gives rise to $7$ different kinds of biregular maps, and we now give two examples.
\smallskip

Perhaps the most well known kind of biregularity arises from the subgroup of a full triangle group formed by orientation-preserving automorphisms, leading to a widely studied class of maps known as orientably-regular \cite{Sir-BCC} or rotary \cite{Wil}. We note that orientability of a map is equivalent to admitting local orientations at vertices that are consistent when passing between incident vertices along an edge on the carrier surface of the map. An opposite phenomenon, namely, when a map of type $(k,\ell)$ admits an assignment of local orientations at vertices which disagree when traversing along an arbitrary edge on the surface and, at the same time, exhibits the highest `level of symmetry' with respect to this property, was studied in \cite{BCS}. These maps, called bi-rotary in \cite{BCS}, arise as normal torsion-free quotients of a different index-two subgroup of the full $(2,k,\ell)$-triangle group.
\smallskip

In this paper we will focus on still another kind of biregular maps, those admitting an action of the map automorphism group with two orbits on edges, accordingly called {\em edge-biregular}. It turns out that edge-biregular maps of type $(k,\ell)$ arise as torsion-free normal quotients from a {\em unique} index-two subgroup of the full $(2,k,\ell)$-triangle group. In such a map, edges emanating from a vertex in their cyclic order on the carrier surface must necessarily alternate between the two orbits, so that $k$ must be even. Similarly, $\ell$ must also be even.
\smallskip

Edge-biregular maps have been investigated in great detail in \cite{Jea}, including their classification on surfaces of non-negative Euler characteristic and also a classification of such maps with a dihedral automorphism group. Our aim here is to complement the results of \cite{Jea} by deriving a classification of edge-biregular maps on surfaces with negative prime Euler characteristic, extending thereby also the existing classification results for fully regular maps \cite{BNS} and bi-rotary maps \cite{BCS} on these surfaces.
\smallskip

To this end, in section 2 we introduce biregular maps in the wider context of symmetric maps of a given type and establish notation and some basic facts about edge-biregular maps. Section 3 contains a summary of our results on classification of edge-biregular maps on surfaces of negative prime Euler characteristic and a setup of proofs; these are subsequently given in sections 4 and 5. Finally section 6 contains further observations and remarks.

\section{Algebra of edge-biregular maps}\label{sec:alg}

It is well known (see for example \cite{JoSi} or \cite{BrSi}, or the survey \cite{Sir-BCC}) that fully regular maps of a given type $(k,\ell)$ can be identified with torsion-free normal quotients of the full $(2,k,\ell)$-triangle group with presentation
\begin{equation}\label{eq:Tkl}
\langle \; R_0,\, R_1,\, R_2 \, | \, R_0^2,\, R_1^2,\, R_2^2,\, (R_0R_2)^2,\ (R_1R_2)^k,\, (R_0R_1)^\ell \, \rangle\ 
\end{equation}
This group is isomorphic to the full automorphism group of a {\em universal $(k,\ell)$-tessellation} of a simply connected surface, where the tiles are regular $l$-gons, and the valency of the underlying graph is $k$. One usually considers a universal $(k,\ell)$-tessellation also as a geometrically regular tiling of a sphere, a Euclidean plane, or a hyperbolic plane, depending on whether the value of $1/k + 1/\ell$ is greater than, equal to, or smaller than $1/2$.
\smallskip

As alluded to in the introduction, biregular maps arise as smooth, that is torsion-free, normal quotients of subgroups of the full $(2,k,\ell)$-triangle groups of index two. Any such subgroup is completely determined by a non-empty subset of $\{R_0, R_1, R_2\}$, consisting of the involutory generators {\em not} contained in the subgroup. There are $7$ such subsets, and the number of resulting subgroups of index two depends on the parities of $k$ and $\ell$. The index two subgroup avoiding each of $R_0,R_1,R_2$ occurs in every one of the triangle groups, regardless of parity of $k$ and $\ell$. If both $k$ and $\ell$ are odd then this is the only subgroup of index two in the full $(2,k,\ell)$-triangle group. If exactly a given one of $k,\ell$ is even, there are two further index two subgroups, and so including the aforementioned subgroup, $3$ such subgroups in total. If both $k$ and $\ell$ are even, the full $(2,k,\ell)$-triangle group contains $7$ subgroups of index two.
\smallskip

The subgroup of \eqref{eq:Tkl} of index two containing none of the three involutory generators, informally denoted $\langle \overline{R}_0, \overline{R}_1, \overline{R}_2\rangle$ here, is known as the (ordinary) $(2,k,\ell)$-triangle group comprising all orientation-preserving automorphisms of a universal $(k,\ell)$-tessellation. It is generated, for instance, by the products $U=R_2R_1$ and $V=R_1R_0$ representing rotations of the tessellation about a vertex and about the centre of a face incident to the vertex, and has (irrespective of parity of $k,\ell$) a presentation of the form
\begin{equation}\label{eq:Tklor}
\langle \overline{R}_0,\, \overline{R}_1,\, \overline{R}_2\,\rangle = \langle U,\, V\, |\, U^k,\, V^\ell,\, (UV)^2\,\rangle
\end{equation}
Smooth normal quotients of this subgroup give rise to orientably-regular maps in the terminology of \cite{Sir-BCC}, also known as rotary maps in \cite{Wil}. The second type of biregular maps mentioned in the introduction, the bi-rotary maps of \cite{BCS}, arise as smooth quotients of the index-two subgroup $\langle R_0,\, \overline{R}_1,\, \overline{R}_2\,\rangle$ of \eqref{eq:Tkl} containing $R_0$ but avoiding $R_1$ and $R_2$, for $\ell$ even. Using the notation $A=R_0$, $Z=R_1R_2$, and the square brackets for a commutator of two elements as usual, this subgroup admits a presentation
\begin{equation}\label{eq:TklAZ}
\langle R_0,\, \overline{R}_1,\, \overline{R}_2\,\rangle = \langle A,\,Z\, |\, A^2,\, Z^k,\, [A,Z]^{\ell/2}\, \rangle
\end{equation}

In this paper we will focus on biregular maps corresponding to torsion-free normal quotients of the subgroup $\langle R_0,\, \overline{R}_1,\, R_2\,\rangle$ of the full $(2,k,\ell)$-triangle group \eqref{eq:Tkl} for both $k$ and $\ell$ even, distinguished by inclusion of $R_0$ and $R_2$ and avoidance of $R_1$. It may be checked that the subgroup $\langle R_0,\, \overline{R}_1,\, R_2\,\rangle$ is the only one among the seven index-two subgroups of \eqref{eq:Tkl} inducing {\em two} edge orbits on a universal $(k,\ell)$-tessellation for $k,\ell$ even. With the help of new variables $X=R_0$, $Y=R_2$, $S=R_1R_0R_1$ and $T=R_1R_2R_1$ it can be shown that a complete presentation of this subgroup is
\begin{equation}\label{eq:XYST}
\langle R_0,\, \overline{R}_1,\, R_2\,\rangle = \langle \, X,\, Y,\, S,\, T\, | \, X^2,\, Y^2,\, S^2,\, T^2,\, (XY)^2,\, (ST)^2,\, (TY)^{k/2},\, (SX)^{\ell/2}\, \rangle
\end{equation}
If $N$ is a torsion-free normal subgroup of $\langle R_0,\, \overline{R}_1,\, R_2\,\rangle$ of finite index, the corresponding quotient group $H=\langle R_0,\, \overline{R}_1,\, R_2\,\rangle/N$ has a presentation of the form
\begin{equation}\label{eq:H}
H = \langle \, x,\, y,\, s,\, t\, | \, x^2,\, y^2,\, s^2,\, t^2,\, (xy)^2,\, (st)^2,\, (ty)^{k/2},\, (sx)^{\ell/2},\, \ldots \, \rangle
\end{equation}
with $x=XN$, $y=YN$, $s=SN$ and $t=TN$. By the correspondence between maps and groups, any finite group $H$ together with a presentation as in \eqref{eq:H} is a group of automorphisms of a finite biregular map, the biregularity of which is derived from the index-two subgroup $\langle R_0,\, \overline{R}_1,\, R_2\,\rangle$ of \eqref{eq:Tkl}. Maps exhibiting this kind of biregularity were referred to in \cite{Jea} as {\em edge-biregular}.
\smallskip

The obvious motivation for this terminology comes from the already mentioned two-orbit action of the subgroup $\langle R_0,\, \overline{R}_1,\, R_2\,\rangle$ on edges of a universal $(k,\ell)$-tessellation. By taking a torsion-free normal quotient as in \eqref{eq:H} this two-orbit action projects onto edges of the quotient map as shown in Fig. \ref{fig1}.
\smallskip

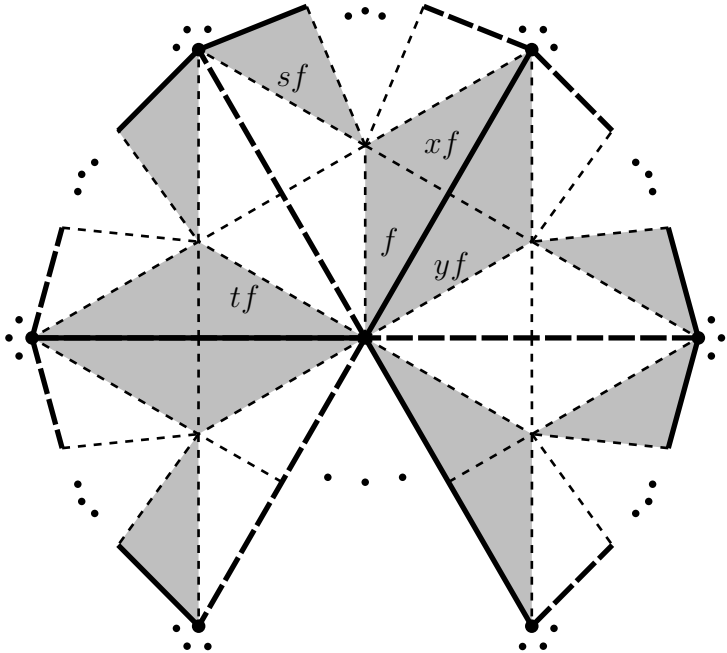
\begin{figure}[h]
\centering
\begin{tikzpicture}
[scale=0.64]

 \fill [gray, opacity=0.5] (-150:4) -- (-120:2*3.46) -- (-140:6.7) -- (-150:4) ;
    
    \fill[gray, opacity=0.5] (0,0) -- (30:4) -- (60:2*3.46) -- (90:4) -- (0,0);
    
  \fill[gray, opacity=0.5] (0,0) -- (150:4) -- (-180:2*3.46) -- (-150:4) -- (0,0);
  
  \fill [gray, opacity=0.5] (0,0) -- (-30:4) -- (-60:2*3.46) -- (0,0);
   
   \fill[gray, opacity=0.5] (150:4) -- (120:2*3.46) -- (140:6.7) -- (150:4) ;
    
    \fill [gray, opacity=0.5] (100:7) -- (120:2*3.46) -- (90:4) -- (100:7);
    
\fill [gray, opacity=0.5] (0:2*3.46) -- (20:6.7) -- (30:4);
\fill [gray, opacity=0.5] (0:2*3.46) -- (-20:6.7) -- (-30:4);

  \draw [line width=2pt, dash pattern={on 10pt off 2pt}] (2*3.46,0) -- (0,0) (0,0) -- (120:2*3.46);
 \draw [line width=2pt] (-2*3.46,0) -- (0,0) (0,0) -- (60:2*3.46);

\draw [dashed, line width=1pt] (90:4) -- (30:4);
\draw [dashed, line width=1pt] (90:4) -- (150:4);

\draw [dashed, line width=1pt] (0:2*1.73) -- (30:4);

\draw [dashed, line width=1pt] (180:2*1.73) -- (150:4);

\draw [dashed, line width=1pt] (90:4) -- (80:7);
\draw [dashed, line width=1pt] (90:4) -- (100:7);
 \draw [line width=2pt] (120:2*3.46) -- (100:7);
  \draw [line width=2pt, dash pattern={on 10pt off 2pt}] (60:2*3.46) -- (80:7) ;

\draw [dashed, line width=1pt] (30:4) -- (40:6.7);
 \draw  [line width=2pt, dash pattern={on 10pt off 2pt}] (60:2*3.46) -- (40:6.7);
 \draw  [line width=2pt] (120:2*3.46) -- (140:6.7) ;

 \draw[dashed, line width=1pt] (150:4) -- (120:2*3.46) -- (140:6.7) -- (150:4) ;

  \draw [dashed, line width=1pt] (0,0) -- (30:4) -- (60:2*3.46) -- (90:4) -- (0,0);

 \draw [dashed, line width=1pt]  (30:4) -- (0:2*3.46) (0,0) -- (-30:4) (-30:4)--(0:2*3.46) (-30:4)--(-60:2*3.46) (90:4)--(120:2*3.46)  (-150:4)--(-120:2*3.46);

  ***

\draw [line width=2pt, dash pattern={on 10pt off 2pt}] (-2*3.46,0) -- (0,0) (0,0) -- (-120:2*3.46);
 \draw [line width=2pt] (0,0) -- (-60:2*3.46);

\draw [dashed, line width=1pt] (-60:3.46) -- (-30:4);
\draw [dashed, line width=1pt] (-120:3.46) -- (-150:4);

\draw [dashed, line width=1pt] (0:2*1.73) -- (-30:4);
  \draw [dashed, line width=1pt] (0,0) -- (150:4) -- (-180:2*3.46) -- (-150:4) -- (0,0);

\draw [dashed, line width=1pt] (-180:2*1.73) -- (-150:4);

\draw [dashed, line width=1pt] (-30:4) -- (-40:6.7);
\draw [dashed, line width=1pt] (-150:4) -- (-140:6.7);
 \draw  [line width=2pt, dash pattern={on 10pt off 2pt}] (-60:2*3.46) -- (-40:6.7);
 \draw  [line width=2pt] (-120:2*3.46) -- (-140:6.7) ;

   \draw [line width=2pt, dash pattern={on 10pt off 2pt}] (-2*3.46,0) -- (160:6.7) (-2*3.46,0) -- (-160:6.7);
 \draw [line width=2pt] (2*3.46,0) -- (20:6.7) (2*3.46,0) -- (-20:6.7) ;

\draw [dashed, line width=1pt] (-20:6.7) -- (330:4);
\draw [dashed, line width=1pt] (20:6.7) -- (30:4);

\draw [dashed, line width=1pt] (160:6.7) -- (150:4);
\draw [dashed, line width=1pt] (200:6.7) -- (210:4);

   \fill (0,0) circle (4pt);
   \foreach \x in {0,...,5}
   \fill (60*\x:2*3.46) circle (4pt);

   \draw (270:3) node[circle,fill,inner sep=1pt]{};
  \draw (255:3) node[circle,fill,inner sep=1pt]{};
  \draw (285:3) node[circle,fill,inner sep=1pt]{};

     \foreach \x in {-2,...,2}
   \draw (90+60*\x:6.8) node[circle,fill,inner sep=1pt]{}
   (87+60*\x:6.7) node[circle,fill,inner sep=1pt]{}
 (93+60*\x:6.7) node[circle,fill,inner sep=1pt]{};

   \foreach \x in {0,...,5}
   \draw (60*\x:7.4) node[circle,fill,inner sep=1pt]{}
  (3+60*\x:7.2) node[circle,fill,inner sep=1pt]{}
(-3+60*\x:7.2) node[circle,fill,inner sep=1pt]{};

 \draw (0.5,2) node {$f$};
   \draw (1,4) node[above, right] {$xf$};
    \draw (1.2,1.5) node[right] {$yf$};
     \draw (-1.5,5.4) node {$sf$};
     \draw (-2.5,0.8) node {$tf$};

 \end{tikzpicture}
\caption{The images under $x,y,s,t$ of the distinguished flag $f$ of an edge-biregular map}
\label{fig1}
\end{figure}

In order to introduce a `reference system' within the map on the diagram one may fix and shade an arbitrarily chosen flag $f$, let $x$ and $y$ represent the reflections of $H$ in the sides of $f$ as depicted, and call the quadruple of flags $f$, $xf$, $yf$ and $xyf$ surrounding an edge `distinguished'. The pattern of the shaded flags in this diagram demonstrate how the group of automorphisms $H$ of an edge-biregular map `spreads around' the distinguished quadruple and hence acts transitively on faces and vertices. At the same time, Fig. \ref{fig1} displays two orbits of $H$ on edges, those which are bold in one orbit, with dashed edges indicating the other orbit; notice also the two $H$-orbits on flags formed by quadruples of shaded and unshaded flags surrounding respectively bold and dashed edges.
\smallskip

Observe that it is a choice, consistent with the choice of notation, which orbit of edges is coloured bold, and which is dashed. As shown in Fig. \ref{fig1}, the generators $x$ and $y$ correspond to automorphisms acting locally as reflections respectively along and across the distinguished bold edge. Meanwhile $s$ and $t$ correspond to reflections respectively along and across the distinguished dashed edge, that is the dashed edge which is surrounded by the quadruple of unshaded flags sharing a boundary with $f$. We call this, and the corresponding presentation \eqref{eq:H}, the {\em canonical form} of the edge-biregular map $M$ and denote it $M = (H; x,y,s,t)$. A pair of edge-biregular maps $M=(H;x,y,s,t)$ and $M'=(H';x',y',s',t')$ given in the canonical form are {\em isomorphic} if there is a group isomorphism $H\to H'$ taking $z$ onto $z'$ for every $z\in \{x,y,s,t\}$.
\smallskip

Whether an edge-biregular map $M=(H;x,y,s,t)$ is also fully regular will depend on the existence (or otherwise) of an involutory automorphism of $M$ lying outside $H$ and fusing the two $H$-orbits of edges together. By the above, this condition is equivalent to the existence of an automorphism of the {\em group} $H$ which interchanges $x$ with $s$ and $y$ with $t$. The map $M= (H; x,y,s,t)$ is thus fully regular if and only if $M$ and $M' = (H; s,t,x,y)$ are isomorphic as maps, otherwise we say these maps are {\sl twins}.
\smallskip

To avoid unnecessary work, all our forthcoming results will be up to duality. The dual map $M^*$ of an edge-biregular map $M = (H; x,y,s,t)$ is also an edge-biregular map and is formed by interchanging $x$ with $y$ and $s$ with $t$ in the presentation \eqref{eq:H} for $H$, thereby swapping the vertices with the faces and vice versa, to give $M^* = (H; y,x,t,s)$. If $M$ is isomorphic to $M^*$ the map is {\em self-dual}, which (by the map isomorphism condition) is equivalent to the existence of an automorphism of the {\em group} $H$ swapping $x$ with $y$ and $s$ with $t$.
\smallskip

Except for the case of characteristic $-2$, all our edge-biregular maps $M=(H;x,y,s,t)$ will be carried by non-orientable surfaces. Since each canonical generator of $H$ corresponds to a reflection on the carrier surface, it follows that $M$ is supported by an orientable surface if and only if every relator in the presentation of $H$ has an even length in terms of $x,y,s,t$. This is easily seen to be equivalent to the statement that the carrier surface of $M$ is non-orientable if and only if $H$ is generated by any three products, of two involutions each, provided that every involution out of $x,y,s,t$ appears in at least one product.
\smallskip

Having assumed that our edge-biregular maps $M=(H;x,y,s,t)$ arise from the index-two subgroup $\langle R_0,\, \overline{R}_1,\, R_2\,\rangle$ of full triangle groups by torsion-free normal subgroups implies that all the four canonical involutory generators of $H$ are {\em distinct}. For completeness we address also the situations when some of the generators in the set $\{x,y,s,t\}$ are either trivial or equal to each other. These are commonly known as `degeneracies' and have been treated in detail in \cite{Jea}; they give rise to maps on surfaces with boundary or to maps with semi-edges or to maps of spherical type with $2 \in \{k,l\}$. Recall that a semi-edge has one of its endpoints attached to a vertex but the other one is `dangling' and not incident to any vertex; we will regard edges and semi-edges as different objects.
\smallskip

While we are not considering maps on surfaces with boundary here, edge-biregular maps with semi-edges are relevant for us. Leaving the trivial case of a semi-star (a spherical one-vertex map with some number of attached semi-edges) aside, suppose that an edge-biregular map of type $(k,\ell)$ with both entries even contains a semi-edge as well as an edge. By \cite{Jea}, edges and semi-edges must then alternate around each vertex on the carrier surface of the map, and consequently the deletion of all semi-edges results in a fully regular map of type $(k/2,\ell/2)$. Conversely, every edge-biregular map containing both an edge and a semi-edge arises from a fully regular map that is not a semi-star by inserting a semi-edge into each corner.
\smallskip

To arrive at an algebraic explanation of this edge/semi-edge phenomenon in edge-biregular maps, consider such a map $M=(H; x,y,s,t)$ in a canonical form and assume that all the dashed edges in Fig. \ref{fig1} would collapse to semi-edges. This is equivalent to identifying the flags marked $sf$ and $tf$ into a single flag, which means regarding $s$ and $t$ as identical automorphisms. It now follows from \eqref{eq:H} that, up to twinness, an edge-biregular map $M=(H; x,y,s,t)$ of type $(k,\ell)$ for even $k$ and $\ell$, containing both edges and semi-edges, may be identified with its group of automorphisms $H$ presented in the form
\begin{equation}\label{eq:Hsemi}
H = \langle \, x,\, y,\, s,\, t\, | \, x^2,\, y^2,\, s^2,\, t^2,\, (xy)^2,\, st,\, (ty)^{k/2},\, (sx)^{\ell/2},\, \ldots \, \rangle\ 
\end{equation}
Note the difference from \eqref{eq:H} in the power at the product $st$, implying $s=t$ in \eqref{eq:Hsemi}. Conversely, let a fully regular map $M$ of type $(k/2,\ell/2)$ for $k,\ell$ even and distinct from a semi-star be given by its (full) automorphism group, that is, by a torsion-free normal quotient of \eqref{eq:Tkl} presented in the form
\begin{equation}\label{eq:Tkl-r}
\langle \; r_0,\, r_1,\, r_2 \, | \, r_0^2,\, r_1^2,\, r_2^2,\, (r_0r_2)^2,\ (r_1r_2)^{k/2},\, (r_0r_1)^{\ell/2} \, \ldots \, \rangle\ 
\end{equation}
Then, a presentation of the group $H$ as in \eqref{eq:Hsemi} for the corresponding edge-biregular map of type $(k,\ell)$ arising from $M$ by inserting a semi-edge into every corner of $M$ is obtained from the one in \eqref{eq:Tkl-r} simply by letting $x=r_0$, $y=r_2$, and $s=t=r_1$.
\smallskip

\section{Results, and proofs setup}\label{sec:res}

In this section we state our classification results for edge-biregular maps on surfaces of negative prime Euler characteristic. We note that a classification of such maps on surfaces with a non-negative Euler characteristic (a sphere, a projective plane, a torus and a Klein bottle) was derived in \cite{Jea}.
\smallskip

As already alluded to, edge-biregular maps on surfaces without boundary split into two families, according to whether they contain semi-edges or not. Also, for the purpose of this section we may disregard the semi-star maps. Let us first consider edge-biregular maps containing both edges as well as semi-edges. By the facts summed up at the end of the previous section, edge-biregular maps of type $(k,\ell)$ for even $k$ and $\ell$ containing both edges as well as semi-edges and carried by a particular surface are in a one-to-one correspondence with fully regular maps of type $(k/2,\ell/2)$ on the same surface. It follows that a classification of all fully regular maps on a particular surface automatically implies a classification of edge-biregular maps with semi-edges on the same surface, including presentations of the corresponding groups of automorphisms. And since a classification of fully regular maps on surfaces of negative prime Euler characteristic is available from \cite{BNS}, a corresponding classification of edge-biregular maps with semi-edges on these surfaces follows by the outlined procedure. In the interest of saving space we will not go into any further detail and we also omit a formal statement of the related presentations, referring the reader to \cite{BNS} or to \cite{Sir-BCC}.
\smallskip

We now pass onto classification of edge-biregular maps without semi-edges on surfaces with negative prime Euler characteristic. From the work in \cite{Dua} and \cite{Jea}, we know that if such a map $M=(H;x,y,s,t)$ is given in a canonical form with the group of automorphisms $H$ presented as in \eqref{eq:H}, then all the four generators are {\em non-trivial} and {\em distinct}. We state our main results separately for odd primes $p$ (in which case, of course, all the maps are carried by non-orientable surfaces) and then for $p=2$.

\begin{thm}\label{thm:main-odd}
A surface of negative odd prime Euler characteristic $\chi = -p$ supports, up to duality and twin maps, only the following pairwise non-isomorphic edge-biregular maps $M = (H;x,y,s,t)$ with no semi-edges, given in their canonical forms:
\begin{enumerate}
\item
A single-vertex map of type $(4(p+1),4)$ with $H=H_{p(1)}= \langle y,t \rangle$ isomorphic to the dihedral group of order $4(p+1)$ with canonical presentation
\begin{equation*}
H_{p(1)}= \langle \, x,y,s,t \, | \, x^2,\,y^2,\,s^2,\,t^2,\,(xy)^2,\,(st)^2,\, xys, \, s(yt)^{p+1}\rangle
\end{equation*}
\item
A two-vertex map of type $(2(p+2),4)$ with $H= H_{p(2)}=\langle x,t \rangle$ isomorphic to the dihedral group of order $4(p+2)$ and canonical presentation
\begin{equation*}
H_{p(2)} = \langle \, x,y,s,t \, | \, x^2,\,y^2,\,s^2,\,t^2,\,(xy)^2,\,(st)^2,\, xys, \, s(xt)^{p+2}\rangle
\end{equation*}
\item
The maps $M_{p,j}$ of type $(k, \ell)$ where $k=4\kappa$ and $\ell=2\lambda$ for odd and relatively prime $\kappa$ and $\lambda$, such that $p=2\kappa\lambda -2\kappa -\lambda$, where $H=H_{p,j}$ is of order $kl/2$ and presented as follows, for any positive integer $j < \lambda$ such that $j^2\equiv 1$ {\rm mod} $\lambda$, with $a=(j-1)(\lambda+1)/2$, and $u=sx$, $v=ty$:
\begin{equation*}
H_{p,j} = \langle \, x,y,s,t \, | \, x^2,\,y^2,\,s^2,\,t^2,\,(xy)^2,\,(st)^2,\,u^{\lambda},\,
v^{2\kappa},\, [s,v^2],\,[x,v^2],\, tutu^j,\, v^\kappa u^a s\, \rangle
\end{equation*}
\item
When $p\equiv 5$ {\rm mod} $9$, that is when $p=9m-4$ we have a map of type $(8,6m)$ and the presentation of the corresponding group $H=H_p$ which has order $24m$ is of the form
\begin{equation*}
H_p = \langle x,y,s,t\, | \, x^2,\,y^2,\,s^2,\,t^2,\,(xy)^2,\,(st)^2,\,(sx)^{3m},\,(ty)^4,\, (sxy)^2t,\,txty \rangle
\end{equation*}
\item
When $p=3$, we have a map of type $(4,6)$ and the corresponding group $H=H_{(3)} \cong D_6 \times D_6$ has canonical presentation
\begin{equation*}
H_{(3)} = \langle x,y,s,t\, | \, x^2,\,y^2,\,s^2,\,t^2,\,(xy)^2,\,(st)^2,\,(ty)^2,\,(sx)^3,\, (xyt)^3, \, (sty)^3, \, (xyst)^2  \rangle
\end{equation*}
\end{enumerate}
Moreover, the only fully regular map in the above list is $H_{(3)}$.
\end{thm}

\begin{thm}\label{thm:main-even}
A surface of Euler characteristic $\chi = -2$ supports, up to duality and twin maps, only the following non-isomorphic edge-biregular maps $M = (H;x,y,s,t)$ with no semi-edges:
\begin{enumerate}
\item
One self-dual map of type $(8,8)$ with a single vertex and a single face, with the group $H=H_{2,1}$ presented as
\[ H_{2,1} = \langle \; x, y, s, t \; |  \; x^2, y^2, s^2, t^2, (xy)^2, (st)^2, (ty)^{4}, (sx)^{4}, ysxs, txsx, xtyt, syty \rangle \cong D_8  \]
\item
One map of type $(4,12)$ with three vertices and a single face, with the group $H=H_{2,2}$ having canonical presentation
\[ H_{2,2} = \langle \; x, y, s, t \; |  \; x^2, y^2, s^2, t^2, (xy)^2, (st)^2, (ty)^{2}, (sx)^{6}, xyt, t(sx)^3 \rangle \cong D_{12}  \]
\item
One self-dual map of type $(6,6)$ with two vertices and two faces, with the group $H=H_{2,3}$ being as follows
\[ H_{2,3} = \langle \; x, y, s, t \; |  \; x^2, y^2, s^2, t^2, (xy)^2, (st)^2, (ty)^{3}, (sx)^{3}, styx \rangle \cong D_{12}  \]
\item
Six maps of type $(4,8)$ with four vertices and two faces, with groups $H_{2,i}\cong D_8\times C_2$ for $i=4,5,6,7,8,9$ presented in the form
\begin{align*}
H_{2,4}&= \langle \; x, y, s, t \; |  \; x^2, y^2, s^2, t^2, (xy)^2, (st)^2, (ty)^{2}, (sx)^{4}, (ys)^4, (tx)^2, 
ysxs \rangle  \\
H_{2,5}&= \langle \; x, y, s, t \; |  \; x^2, y^2, s^2, t^2, (xy)^2, (st)^2, (ty)^{2}, (sx)^{4}, (ys)^2, (tx)^2, y(sx)^2 \rangle \\
H_{2,6}&= \langle \; x, y, s, t \; |  \; x^2, y^2, s^2, t^2, (xy)^2, (st)^2, (ty)^{2}, (sx)^{4}, (ys)^2, (stx)^2, y(sx)^2 \rangle \\
H_{2.7}&= \langle \; x, y, s, t \; |  \; x^2, y^2, s^2, t^2, (xy)^2, (st)^2, (ty)^{2}, (sx)^{4}, (ys)^2, (tx)^2, 
ty(sx)^2 \rangle \\
H_{2.8}&= \langle \; x, y, s, t \; |  \; x^2, y^2, s^2, t^2, (xy)^2, (st)^2, (ty)^{2}, (sx)^{4}, (ys)^2, (stx)^2, 
tyxsx \rangle \\
H_{2.9}&= \langle \; x, y, s, t \; |  \; x^2, y^2, s^2, t^2, (xy)^2, (st)^2, (ty)^{2}, (sx)^{4}, (ys)^2, (stx)^2, 
tys \rangle 
\end{align*}
\item

Three maps of type $(4,6)$ with six vertices and four faces, with groups of automorphisms $H_{2,i}$ for $i=10,11,12$ given by
\begin{align*}
H_{2,10}&= \langle \; x, y, s, t \; |  \; x^2, y^2, s^2, t^2, (xy)^2, (st)^2, (ty)^{2}, (sx)^{3}, t(sy)^2, y(xt)^2 \rangle \cong S_4 \\
H_{2,11}&= \langle \; x, y, s, t \; |  \; x^2, y^2, s^2, t^2, (xy)^2, (st)^2, (ty)^{2}, (sx)^{3}, (ysx)^2,(tx)^2 \rangle \cong D_6\times V_4 \\
H_{2,12}&= \langle \; x, y, s, t \; |  \; x^2, y^2, s^2, t^2, (xy)^2, (st)^2, (ty)^{2}, (sx)^{3}, (ys)^2,
(tx)^2 \rangle \cong D_6\times V_4
\end{align*}
\end{enumerate}
The maps supported by the groups $H_{2,1}$, $H_{2,3}$, $H_{2,4}$, $H_{2,7}$, $H_{2,11}$ and $H_{2,12}$ are orientable while the other six are not, and the only fully regular ones out of the twelve are those supported by the groups $H_{2,1}$, $H_{2,3}$, $H_{2,7}$, $H_{2,10}$ and $H_{2,12}$.
\end{thm}
 
We make no claim of a lack of redundancy in these presentations - in fact quite the opposite: some of the above presentations have unnecessary relators which have been retained. This is deliberate in order to better demonstrate the interplay between the canonical generators, and also to make evident the presence, or absence, of full regularity or self-duality in the underlying map.

In order to set the stage for proofs in the next three sections, throughout we let $M=(H;x,y,s,t)$ be an edge-biregular map of type $(k,\ell)$ for even $k$ and $\ell$, with no semi-edges, and of characteristic $-p$ (meaning that the carrier surface of $M$ has Euler characteristic $\chi=-p$). All our working will be up to duality, so that we may without loss of generality assume that $k\le\ell$. We also know by \cite{Dua} and \cite{Jea} that in this situation the four generating involutions of the associated group of automorphisms $H$ are mutually distinct. Note that this instantly implies that $k,\ell\ge 4$; moreover, since maps of type $(4,4)$ necessarily live on surfaces with Euler characteristic $0$, we will assume that $k\ge 4$ and $\ell\ge 6$.
\smallskip

Referring to the diagram in Fig. \ref{fig1} which may be assumed to display a fragment of $M$, the stabiliser in $H$ of the distinguished face is $\langle s, x \rangle$, isomorphic to the dihedral group $D_\ell$ of order $\ell$, while the $H$-stabiliser of the distinguished vertex is $\langle t, y \rangle$, isomorphic to the dihedral group $D_k$ of order $k$. The map thus has $\frac{|H|}{\ell}$ faces and $\frac{|H|}{k}$ vertices. The $H$-stabiliser of the distinguished bold and dashed edge, respectively, is $\langle x,y \rangle$ and $\langle s,t \rangle$, in both cases isomorphic to the Klein four-group $V_4\cong C_2\times C_2$; hence the map has $\frac{2|H|}{4}$ edges. The Euler-Poincar\'e formula applied to the number of vertices, edges and faces of $M$ gives $|H| (\frac{1}{k} - \frac{1}{2}	+\frac{1}{\ell}) = \chi = -p$, or, equivalently,
\begin{equation}\label{eq:Eu}
|H| = \nu p \ \ \ \ {\rm where} \ \ \ \ \nu = \nu(k,\ell) = \frac{2k\ell}{k\ell-2(k+\ell)}
\end{equation}

It may be checked that our lower bounds $k\ge 4$ and $\ell\ge 6$ imply $\nu \le 12$. Combining this with \eqref{eq:Eu} one obtains the upper bound $|H| \le 12p$ for the group of automorphisms $H$ of the map $M$ under consideration. Also, by \eqref{eq:H} and distinctness of the involutory generators, the group $H$ contains two distinct subgroups isomorphic to $V_4$; in particular, the order of $H$ must be divisible by $4$.
\smallskip

From this point on we split the proof into four parts. In section \ref{sec:div} we will consider the case when $p$ divides the order of $H$, first addressing all odd primes, and then the special case when $p=2$. Section \ref{sec:nodiv} will deal with the opposite case, and the analysis will be divided according to whether the $2$-part Fitting subgroup of $H$ is cyclic or dihedral.

\section{The case when $p$ divides the order of $H$}\label{sec:div}

Let us recall that we are considering an edge-biregular map $M=(H;x,y,s,t)$ of type $k,\ell$ for even $k,\ell$ such that $k\le \ell$, $k\ge 4$, $\ell\ge 6$, carried by a surface of characteristic $-p$ for some prime $p$. Suppose $p$ is a divisor of $|H|$. This, together with \eqref{eq:Eu} and the facts summed up at the end of section \ref{sec:res}, implies that $|H|=\nu p$ where $\nu\in \{4,6,8,12\}$. Moreover, one may check that in our range for the even entries in the type of our biregular map $M$ one has:
$\nu(k,\ell)=12$ only for $(k,\ell)= (4,6)$;
$\nu(k,\ell)=8$ only for $(k,\ell)=(4,8)$;
$\nu(k,\ell)=6$ only when $p=2$ and for types $(6,6)$ and $(4,12)$;
and finally $\nu(k,\ell)=4$ only for $(k,\ell)$ equal to $(6,12)$ or $(8,8)$, corresponding to the cases when $p=3$ and $p=2$ respectively.

The case when $p=2$ is treated as a special case in the proof of Theorem \ref{thm:main-even} at the end of this section. Henceforth, unless stated otherwise, we assume that $p$ is an odd prime.
We sum up the above observations for odd $p$ in the form of a lemma.

\begin{lem}\label{lem:kl}
If $p$ is an odd prime which divides $|H|$, then either 
$|H|=12p$ for the type $(4,6)$, or
$|H|=8p$ for the type $(4,8)$,
or else $p=3$ and $|H|=4p$ for type $(6,12)$. \hfill $\Box$
\end{lem}

Note that our previous considerations apply also to the situation when $\chi=-1$, that is, Lemma \ref{lem:kl} can be applied also to the exceptional case when $p=1$. In fact, a more detailed look into this case will be useful later. Remembering that we need four distinct involutory generators for $H$ if $|H|\in \{4,8,12\}$ this leaves us with only two options, namely, $|H| = 8$ when $(k,\ell) = (4, 8)$, and $|H| = 12$ when $(k,\ell) = (4, 6)$.
If $|H|=8$ then $\ell=8$ and hence $H \cong D_8$. If, on the other hand, $|H|=12$ then $\ell=6$, which means that $H$ contains a subgroup isomorphic to $D_6$. There is only one such group $H$ of order $12$, namely $D_{12}$. By the classification of \cite{Jea} each of these dihedral groups support an edge-biregular map on a surface with $\chi = -1$. Since we do not need more detailed information about such maps, we just state the following conclusion here:

\begin{lem}\label{lem:chi1}
If $M=(H;x,y,s,t)$ is an edge-biregular map on a surface of Euler characteristic $\chi = -1$, then $H$ is  isomorphic to a dihedral group of order $8$ or $12$. \hfill $\Box$
\end{lem}

We continue by showing that assuming $p\; | \; |H|$ leads to contradictions for $p\ge 13$. To do so we first exclude the existence of maps $M$ with types as above for which the group $H$ would be a semidirect product of a particular form.

\begin{lem}\label{lem:ex}
Let $p$ be an odd prime which does not divide $\kappa$ or $\lambda$ where $\lambda \geq 3$. Suppose $M=(H;x,y,s,t)$ is an edge-biregular map of type $(2\kappa, 2\lambda)$ such that $H \cong C_p \rtimes D_\nu$.
Then $\nu = 2\lambda \equiv 4$ mod $8$, and $\kappa \in \{ 2, \lambda\}$ while the supporting surfaces for these maps have $\chi = p(1-\lambda/2)$ and $\chi = p(2-\lambda)$ respectively.
\end{lem}

\begin{proof}
Suppose that an edge-biregular map of type $(2\kappa,2\lambda)$ exists with $H \cong C_p \rtimes D_\nu$. With the usual notation, for any edge-biregular map we have $\langle x, y \rangle \cong  \langle s, t \rangle \cong V_4$, and hence $4$ must divide $|D_\nu|$. In particular this means that $D_\nu$ has non-trivial centre, and this leaves us with only two congruence classes to consider for $\nu$, namely $0$ mod 8, where the central element of $D_\nu$ is a square, and $4$ mod 8, where the central element of the dihedral group is not a square.

\smallskip

We use additive notation in the first coordinates to indicate the Abelian cyclic group $C_p$ and multiplicative notation for the dihedral part of the semi-direct product, where $\Phi : D_\nu \rightarrow Aut(C_p)$, which maps $\alpha \rightarrow \phi_\alpha$, is the associated homomorphism. Thus each element of $H$, and hence each of the canonical generators $\{ x,y,s,t \}$, will be written $(a, \alpha)$ for some $a \in C_p$ and some $\alpha \in D_\nu$.

 For any involution $\alpha \in D_\nu$, the automorphism $\phi_\alpha$ must have order dividing two. Since $Aut(C_p) \cong C_{p-1}$ is cyclic there are only two such elements, the identity and the unique involution, so for $g$, a generating element of $C_p$, we have $\phi_\alpha$ must either fix $g$ or invert $g$. Hence for any $\beta \in D_\nu$ we will have $\phi_\beta$ being one of these two automorphisms. Also $(a,\alpha)^2 = (a+\phi_\alpha(a) , \alpha^2)$ so the element $(a,\alpha)$ is an involution in $H$ if and only if both $a+\phi_\alpha(a) = 0$ and $\alpha^2 =1$. Note that if $a \neq 0$ then, in order for $(a, \alpha)$ to be an involution, we must have $\phi_\alpha(a)=-a$, that is $\phi_\alpha$ must be inverting.

Triples of commuting involutions, that is the non-trivial elements of a copy of $V_4$ in the group $H$, have the form:
$\{ (a,\alpha), (b, \beta), (a+\phi_\alpha(b), \alpha \beta) \}$ such that $a+\phi_\alpha(b) = b + \phi_\beta (a)$, and where $\alpha$,  $\beta$ and $\alpha \beta $ are all involutions in $D_\nu$, one of which must necessarily be central.
So, if both $a$ and $b$ are non-zero in $C_p$, then $a-b=b-a$, so $2a=2b$, and hence $a=b$, so the third element of the triple must be $(0,\alpha \beta)$. It is important to note at this point that in this case $\alpha \beta$ would fix any $g \in C_p$.
If instead we suppose that $a \neq 0$ and $b=0$ then $a+\phi_\alpha(b) = b + \phi_\beta (a)$ becomes $a=\phi_\beta(a)$, in other words $\phi_\beta$ must fix every element of $C_p$.

\smallskip

For much of this proof we only need consider the action generated by the dihedral part of the elements in question, rather than the element itself, and so we abuse the notation to denote those in the kernel of $\Phi$, which fix the cyclic elements, by ``$+$" and members of the other coset, whose elements invert the cyclic group, by ``$-$". Hence any copy of $V_4$ in $H$ contains three distinct non-identity elements written: $(a,-), (a,-), (0,+) $ for some fixed $a \in C_p$ (which could itself be zero); or, exceptionally, $(0,+), (0,+), (0,+)$.

We now consider what combination of forms can occur within the set of four distinct canonical involutory generators for $H$, namely $\{ x,y,s,t \}$, taking each case in turn.

\smallskip

It will be useful to note that assuming $H$ is generated by any number of involutory elements, each of which has the form $(0,+)$ or $(a,-)$ for a given fixed $a \in C_p$, leads to a contradiction. Indeed since Ker$\Phi$ must have index two in $D_\nu$, given any two elements $\alpha$ and $\beta$ which are not in the kernel, we have $\alpha\beta \in$ Ker$\Phi$. Hence, in this case, the product of elements $(a, \alpha)(a, \beta)$ must have the form $(0, \alpha\beta)$ that is $(0,+)$. Also then, for $\gamma \in$ Ker$\Phi$, the product $(0, \gamma)(a, \alpha)$ must have the form $(0+a, \gamma\alpha)$ that is $(a,-)$, and similarly $(a, \alpha)(0, \gamma)$ must have the form $(a,-)$. Regardless of the orders of any such products, this restrictive set of forms makes it clear that $a$ is just a place-holder marking when the dihedral part of the element is not in the kernel of $\Phi$, and as such this fixed $a$ cannot contribute towards generating $C_p$ in the first coordinate. So, this is in contradiction to $H \cong C_p \rtimes D_\nu$.

\smallskip

We now highlight another property which will be of use:

Since $\langle x,s\rangle$ is a group isomorphic to $D_{2\lambda}$, henceforth we identify the dihedral parts of $\langle x,s \rangle$ with the corresponding elements for a particular copy of $D_{2\lambda} \leq D_\nu$. Now $2\lambda > 4$ so the central involution of the dihedral group $D_\nu$, which we now call $z$, will never occur as the second coordinate of $x$ or $s$ (for otherwise we would have $\langle x,s \rangle \cong V_4$).

If $\lambda$ is odd then $\langle x,s \rangle$ has trivial centre, and every rotation in $D_{2\lambda}$ is a square element in $D_\nu$. This means that either $D_{2\lambda} \leq $ Ker$\Phi$ or (all of the rotations but) none of the involutions in $D_{2\lambda}$ are in Ker$\Phi$. In particular, the canonical involutions generating $\langle s,x \rangle$ must have forms with the same symbol, be that $+$ or $-$, in the second coordinate.

If $\lambda$ is even then $\langle x,s \rangle$ has non-trivial centre and $z \in D_{2\lambda}$, and, according to the congruence class of $\nu$ modulo 8, this $z$ may or may not be a square in $D_\nu$.

\smallskip

Up to twinness of maps, we now divide the argument into two cases.

\smallskip

\noindent The first case to consider is when $x$ has the form $(a, -)$, for some $a$ which may be zero:

Suppose $s$ has the form $(b,-)$. But $sx$, which has order $\lambda$, is then denoted $(b-a,+)$. Now $p$ does not divide the face length $2\lambda$, so this forces $b=a$, and hence $s$ is also an $(a,-)$. Now since $t$ commutes with $s$, and $y$ commutes with $x$, we know that each of $t$ and $y$ will also have forms within the restrictive set $\{ (0,+), (a,-) \}$, contradicting $H \cong C_p \rtimes D_\nu$.

Suppose instead that $s$ has the form $(0,+)$. Thus $\lambda$ must be even and so $z \in D_{2\lambda}$. This splits into two cases, according to the congruence class of $\nu$. 

When $\nu \equiv 0$ mod 8, $z$ is a square element in $D_\nu$, and hence $z \in$ Ker$\Phi$. Remember that a triple of commuting involutions must include an element $(\dots, z)$ which in this case has the form $(0,+)$. Hence the set $\{s, t, st \}$ must be $\{ (0,+), (0,+), (0,+) \}$. In this case, along with $y$, which is either $(0,+)$ or $(a,-)$, our canonical generators are once again all contained within the restrictive set of forms $\{ (0,+), (a,-) \}$, contradicting $H \cong C_p \rtimes D_\nu$.

When $\nu \equiv 4$ mod 8, the centre $z \in D_{2\lambda}$ is not a square element in $D_\nu$, and hence there is the possibility of $z \notin$ Ker$\Phi$, which we will now explore. (If the central element was in the kernel we would be in the same situation as immediately above, and so the map would not exist.) The group $\langle s, x \rangle \cong \langle r,f |  r^\lambda, f^2 , (rf)^2 \rangle =  D_{2\lambda} \leq D_\nu$, can only have an inverting central element if $\lambda = 2(2m+1)$ and $r \notin$ Ker$\Phi$. Up to the choice of notation for $f$, a reflection, and $r$, which generates a cyclic group of order $\lambda$, this happens when exactly one of these two generating involutions $f,rf$ is in the kernel, say $rf\in$ Ker$\Phi$, thereby ensuring $r \notin$ Ker$\Phi$ and hence $z = r^{2m+1} \notin$ Ker$\Phi$. With this assumption $x=(a,f)$ and $s=(0,rf)$, which forces $y \in \{ (a,z) , (0,fz) \}$ and $t \in \{ (b,z) , (b,rfz) \}$. Notice in particular that $t$ has the form $(b,-)$. If $y=(a,z)$ then the non-divisibility of $\kappa$ by $p$ forces $a=b$ and thus the canonical generators would be restricted to a set which cannot generate $H$. So, for such an edge-biregular map to survive, $y=(0,fz)$ and the choices for $t$ with $b \neq a$ would give edge-biregular maps of type $(4, 2\lambda)$ and $(2\lambda, 2\lambda)$ respectively. In particular it is clear that no larger dihedral group than $D_{2\lambda}$ can be found in the second coordinates and so $2\lambda = \nu$. Meanwhile the element $xt$ has the form $(a-b,+)$ and hence the cyclic part $C_p$ of $H$ will also be generated by the set of canonical generators. These maps therefore exist, as do their twins
(and indeed duals) while the characteristic for the corresponding supporting surfaces can be found using the Euler-Poincar\'e formula.

\smallskip

\noindent Until now we have been considering the case when $x$ has the form $(a, -)$, for some $a$ which may be zero, and latterly when we also had $s$ being $(0,+)$. The final case to address is when $x$ has the form $(0,+)$:

To avoid being a twin of the maps considered immediately above, we may now assume that $s$ also has the form $(0,+)$. But then, since $p$ does not divide $\kappa$, the order of $yt$, the two canonical generators $y$ and $t$ introduce a maximum of one new form, say $(a,-)$, to the set of canonical generators. Hence we are once again left with an overly restrictive set of generating forms, contradicting $H \cong C_p \rtimes D_\nu$. 
\end{proof}

We are now in position to exclude primes $p\ge 13$ from consideration in the case of $p$ dividing $|H|$.

\begin{prop}\label{prop:11}
If $M=(H;x,y,s,t)$ is an edge-biregular map, and $p$ is a divisor of $|H|$, then $p\le 11$.
\end{prop}

\begin{proof}
We suppose that $H = \langle x,y,s,t \rangle$ is the group for an edge-biregular map on surface of Euler characteristic $-p$ and that $p \; | \; |H|$ where $p \geq 13$. Combining this with Lemma \ref{lem:kl} we note that the edge-biregular map must have type $(4,8)$ or $(4,6)$ forcing $\nu\in \{8,12\}$ and this implies $|H| = \nu  p \leq 12p < p^2$. Hence there is only one Sylow $p$-subgroup of $H$ and this Sylow subgroup is cyclic, and unique (hence normal) in $H$. 
\smallskip

Since we are restricted to maps of type $(4,8)$ or $(4,6)$, and $ p \geq 13$, the factor group $H / C_p$ forms a smooth quotient and so would correspond to an edge-biregular map of the same type as $H$ on a surface with Euler characteristic $\chi = -1$. Applying now Lemma \ref{lem:chi1} one sees that $H/C_p$ is a dihedral group of order $8$ or $12$. Schur-Zassenhaus now implies  that $H$ is a (non-trivial) semi-direct product $H \cong C_p \rtimes D_\nu$. But this is in contradiction to Lemma \ref{lem:ex}, so we conclude that $p\leq 11$. 
\end{proof}

As the next step we extend Proposition \ref{prop:11} to all odd primes greater than 3, while the case where $p=3$ is dealt with separately in Proposition \ref{prop:three}. The results of the following two propositions, and those of Theorem \ref{thm:main-even}, could be found using a computer, however here we adopt a more classical approach.

\begin{prop}\label{prop:odd}
If $M=(H;x,y,s,t)$ is an edge-biregular map on a surface of Euler characteristic $-p$ for some odd prime $p\geq 5$, then $p$ is not a divisor of $|H|$.
\end{prop}

\begin{proof}
The conclusion of Proposition \ref{prop:11} leaves us with a few small cases we need to address, namely, those listed in Lemma \ref{lem:kl} (together with the corresponding types) for $p \in \{3,5,7,11\}$; here we will handle the three larger primes one by one.
\smallskip

\noindent $\bullet$ {\sl The prime $p=5$, type $(4,8)$ for $|H|=8 \times 5$ and type $(4,6)$ for $|H| = 12 \times 5$.}
\smallskip

For type $(4,8)$ with $|H| = 40$. By Sylow theorems the Sylow $5$-subgroup is normal so $H \cong C_{5} \rtimes D_8$, contrary to Lemma \ref{lem:ex}.
\smallskip

The second type to consider here, $(4,6)$, comes with a group $|H|$ of order $60$. We know $H$ contains a (face stabiliser) subgroup isomorphic to $D_6$, and so $H$ is not Abelian. Suppose that $H$ has a unique Sylow $5$-subgroup, which (by the Schur-Zassenhaus theorem) implies that $H \cong C_5 \rtimes K$ for some group $K$ of order $12$. Out of the five such groups $L$, however, only $D_{12}$ is generated by involutions, and the possibility $H \cong C_5 \rtimes D_{12}$ contradicts Lemma \ref{lem:ex}. It follows that $H$ contains  six Sylow $5$-subgroups and so $H \cong A_5$. Now $A_5$ contains 15 involutions and $5$ Sylow $2$-subgroups, meaning that the intersection of two copies of $V_4$ is either trivial or the two groups are the same. Hence $\langle x,y \rangle = \langle y,t \rangle = \langle s,t \rangle \cong V_4$ implies $|H|=4$, which is a
contradiction.
\smallskip

\noindent $\bullet$ {\sl The prime $p=7$, type $(4,8)$ for $|H|=8 \times 7$ and type $(4,6)$ for $|H| = 12 \times 7$.}
\smallskip

For the type  $(4,8)$ we have $|H|=56$. If the Sylow $7$-subgroup of $H$ is not normal then $H$ contains $8 \times 6 = 48$ elements of order $7$, leaving only eight other elements in the group $H$. These eight elements must form the single (and hence normal) Sylow $2$-subgroup, which is the face stabiliser, $D_\ell \cong D_8$. But then all the involutions of $H$ must lie in $D_\ell$ and so $H$ would not be generated by involutions. So the Sylow $7$-subgroup of $H$ is normal and hence $H \cong C_7 \rtimes D_8$ which contradicts Lemma \ref{lem:ex} again.
\smallskip

The type $(4,6)$ and $|H|=84$ is excluded by Sylow theorems and the Schur-Zassenhaus as for the same type but in the case $p=5$: A group of order $84$ has a unique (normal) Sylow $7$-subgroup, and then the unique possibility is $H \cong C_7 \rtimes D_{12}$, which is impossible by Lemma \ref{lem:ex}.
\smallskip

\noindent $\bullet$ {\sl The prime $p=11$, type $(4,8)$ for $|H|=8 \times 11$ and type $(4,6)$ for $|H| = 12 \times 11$.}
\smallskip

The type $(4,8)$ with $|H|=88$ is excluded by observing that $H$ contains a unique Sylow $11$-subgroup, leaving only the possibility $H \cong C_{11} \rtimes D_8$ that contradicts Lemma \ref{lem:ex}.
\smallskip

The final case to consider is the one of the type $(4,6)$ that comes with a group $H$ of order $132$.
If the Sylow $11$-subgroup in $H$ is normal then, by a similar analysis as done for $p=5$, we would have $H \cong C_{11} \rtimes D_{12}$, which is impossible by Lemma \ref{lem:ex}. So there are $12$ Sylow $11$-subgroups and there are $12 \times 10 = 120$ elements of order $11$. If $H$ contained more than one Sylow $3$-subgroup then there would be $8$ elements of order $3$, leaving room for only one Sylow $2$-subgroup. Now, $H$ is generated by involutions and so the Sylow $2$-subgroup $D_k \cong V_4$ cannot be unique, otherwise we would have $H = D_k$. So, the Sylow $3$-subgroup in $H$, isomorphic to $C_3$, must be unique and hence normal. Thus, in our case with $\ell=6$, the subgroup $\langle sx \rangle \cong C_3$ is normal in $H$, which implies that $ysxy, tsxt \in \{ sx, xs \}$. If we suppose $tsxt=sx$ then $txt=x$ so $t$ commutes with $s$ and $x$ and $y$. But $\langle x,y \rangle$ is a Sylow $2$-subgroup of $H$ and so, by the distinctness of the generators, we would have $t=xy$. Hence $H=\langle sx \rangle \rtimes \langle x,y,t \rangle = \langle sx \rangle \rtimes \langle x,y \rangle$ which has order $12$, not $132$. By symmetry, supposing that $y$ fixes $sx$ leads to a contradiction too. Hence we have $tsxt=xs$ and $ysxy=xs$. These give, in turn, $txt=sxs$ and $ysy=xsx$ which, when combined, give $txtysy=1$. However, recalling that $y$ commutes with both $x$ and $t$, this yields $xs=1$, a contradiction. This completes the proof.
\end{proof}
\smallskip

We now consider the exceptional case when $p=3$, which does indeed yield edge-biregular maps, namely those corresponding to the final part of Theorem \ref{thm:main-odd}.

\begin{prop}\label{prop:three}
If $M=(H;x,y,s,t)$ is an edge-biregular map on a surface of Euler characteristic $-3$, then up to duality and twinness, $H$ is of type $(4,6)$ and has a presentation of the form
\[ H = \langle x,y,s,t | x^2, y^2, s^2, t^2, (xy)^2, (st)^2, (yt)^2, (sx)^3, (xyt)^3, (sty)^3, (xyst)^2 \rangle \cong D_6 \times D_6  \]
\end{prop}

\begin{proof}
When the prime $p=3$, we have the following possibilities: type $(6,12)$ for $|H| = 4 \times 3$; type $(4,8)$ for $|H|=8 \times 3$; and type $(4,6)$ for $|H| = 12 \times 3$.
\smallskip

An edge-biregular map of type $(6,12)$ on this surface would have $|H| = 4 \times 3 = 12$ and so the group $H$ must be the face stabiliser $H=D_l \cong D_{12}$ which is dihedral. Referring to the dihedral classification of edge-biregular maps in \cite{Jea} we can see that such a map does not exist.

For type $(4,8)$ we have $|H| = 24$ and as $H$ is generated by involutions, the Sylow $2$-subgroup $D_\ell \cong D_8$ cannot be normal in $H$. This leaves only two options: $H \cong C_{3} \rtimes D_8$, which is clearly impossible, or $H \cong S_4$. Suppose $H \cong S_4$ is represented as a permutation group on the set $\{1,2,3,4\}$. The non-trivial elements in the three copies of $V_4$ in $S_4$ then form the sets $T_1=\{ (12)(34), (12), (34) \}$, $T_2=\{ (13)(24), (13), (24) \}$, $T_3=\{ (14)(23), (14), (23) \}$ and $T_4=\{ (12)(34), (13)(24), (14)(23) \}$. Since $\langle s,x \rangle = D_\ell \cong D_8$, we may assume, without loss of generality, that $\{ s,x \} = \{ (12)(34), (24) \}$. But then the fact that $\langle y,t \rangle = D_k \cong V_4$ leads us to conclude that for the four {\em distinct} canonical generators we must have $x,y,t,s \in T_2 \cup T_4$. But these two sets only generate $D_8$, not the whole of $S_4$, a contradiction.
\smallskip

We proceed with type $(4,6)$ for a group $H$ of order $36$ and let $K$ be its Sylow $3$-subgroup of order $9$.

Consider the {\em non-trivial and transitive} permutation representation $\pi$ of $H$ on the set $H/K$ of the right cosets of $K$ in $H$, given as follows: To each $h\in H$ we assign a permutation $\pi_h$ of the set $H/K$ mapping the coset $Kx$ onto $Kxh$, for every $x\in H$. Now, as $|H/K|=4$, the representation $\pi$ is a group homomorphism from $H$ into $S_4$, and its kernel is a normal subgroup of $H$ distinct from $H$; note also that $\pi$ cannot be surjective (by divisibility of the source and target by $4$ and $8$). The only proper transitive subgroups of $S_4$ are $A_4$ and the unique subgroup isomorphic to $V_4$, so that $|{\rm Ker}(\pi)|\in \{3,9\}$. But notice that if $|{\rm Ker}(\pi)|=3$, so that $H$ contains a normal subgroup ${\rm Ker}(\pi)\cong C_3$, and ${\rm Im}(\pi)\cong A_4$ then we have an immediate contradiction since $A_4$ is not generated by involutions.

\smallskip

It remains to consider the case when $|{\rm Ker}(\pi)|=9$, which means that the kernel coincides with the Sylow $3$-subgroup $K$ of order $9$ in $H$, and $H\cong K\rtimes V_4$.

Suppose, for a contradiction, that $K=C_9$. Now, $C_3$ is characteristic in $C_9$ and hence normal in $H$. There is only one copy of $C_3$ in $H$, and that is $\langle sx \rangle $. Denoting elements in $H/{\langle sx \rangle }$ with bar notation, we then have $\bar{s} = \bar{x}$ and hence $H/{\langle sx \rangle }$ is generated by three commuting involutions $\bar{x}, \bar{y}$ and $\bar{t}$. This has order at most 8, not the required 12, and so is in contradiction to the order of $H$.

We may consider $C_3^2$ as a two dimensional vector space over $GF(3)$ with elements written as vectors. The automorphism group for $C_3^2$ is $GL(2,3)$ which is known to have just one conjugacy class of subgroups isomorphic to $V_4$.
Thus we let
$
 x= ( \underline{x}, x') , \quad
  y= ( \underline{y}, y') , \quad
   t= ( \underline{t}, t') , \quad
    s= ( \underline{s}, s') 
$
where the underlined vectors in the first coordinates are elements of $C_3^2$, and $x',y',t',s' \in V_4 \leq GL(2,3)$. 
Now, since $sx$ has order 3, it is clear that $s'x'$ must be the identity, and so $x'=s'$. Also, since the products $xy$, $yt$ and $ts$ must all have order 2, we certainly have $x' \neq y' \neq t' \neq x'$, and so $t' = x'y'$.

Suppose the homomorphism $\Phi : V_4 \to GL(2,3)$ associated with the semi-direct product is not injective. In any case, the kernel of $\Phi$ cannot be the whole of $V_4$, since this would yield a direct product $H \cong C_3^2 \times V_4$ which certainly cannot be generated by involutions. The only remaining option is for two of the elements in $V_4$ to be mapped to a given involution $\alpha \in GL(2,3)$, while their product is mapped to the identity.
Now $x' \notin$ Ker$\Phi$, otherwise $x$ and $s$ being involutions would then force $\underline{x} = \underline{s} = \underline{0}$ which is absurd since $xs$ must have order 3.
So, up to twinness, we may assume, $x'y' \in$ Ker$\Phi$. But then this would force $\underline{t} = \underline{0}$, meaning $t$ is central in the group $H$. Now, the elements $x$ and $y$ are involutions so $\alpha$ acts to invert any non-zero parts of $\underline{x}$ and $\underline{y}$. Also, $x$ and $y$ commute so $\underline{x}+\alpha(\underline{y}) = \underline{y}+\alpha(\underline{x})$, that is $\underline{x} - \underline{y} = \alpha(\underline{x} - \underline{y})$, and  this forces $\underline{x}=\underline{y}$. Now we have in fact $t=xy$ and so $H = \langle x,t,s \rangle = \langle x,s \rangle \langle t \rangle \cong D_6 \times C_2$, a contradiction.

We have now shown that the homomorphism mapping from $V_4$ to the associated actions in $GL(2,3)$ must be injective, since in any other case we could not generate the whole group $H \cong C_3^2 \rtimes V_4$ just with involutions. So we may now identify $\{ x',y',t',s' \}$ with their images in $GL(2,3)$, choosing the canonical copy of $V_4$ in $GL(2,3)$, which consists of the diagonal matrices.

In the interests of generating the whole group $H \cong C_3^2 \rtimes V_4$, it can be checked that $ x'=s'=-I$. To satisfy the other known order properties for the products of pairs of generating involutions, up to twinness, we find we are restricted to the set of canonical involutions as follows:
\begin{small}
$$
 x= ( \begin{pmatrix} x_1 \\ x_2 \end{pmatrix}, \begin{bmatrix}
-1&0 \\ 0& -1
\end{bmatrix}) , \quad
  y= ( \begin{pmatrix} 0 \\ x_2 \end{pmatrix}, \begin{bmatrix}
1&0 \\ 0& -1
\end{bmatrix}) , \quad
   t= ( \begin{pmatrix} s_1 \\ 0 \end{pmatrix}, \begin{bmatrix}
-1&0 \\ 0& 1
\end{bmatrix}) , \quad
    s= ( \begin{pmatrix} s_1 \\ s_2 \end{pmatrix}, \begin{bmatrix}
-1&0 \\ 0& -1
\end{bmatrix}) 
$$ 
\end{small}
\noindent
where $x_1 \neq s_1$ and $x_2 \neq s_2$. Regardless of whichever allowable choices of values for $x_i$ and $s_i$ are made, it is clear that $\langle xyt , sty \rangle \cong C_3^2$. Also we have $\langle y, t \rangle\cong V_4$ as expected, and hence $H =  \langle x,y,s,t \rangle = \langle xyt , sty \rangle \rtimes \langle y, t \rangle\cong C_3^2 \rtimes V_4$. Also $H = \langle xyt , xy \rangle \times \langle sty, st \rangle \cong D_6 \times D_6$ and a presentation for the resulting map will be as follows:
\[ H = \langle x,y,s,t | x^2, y^2, s^2, t^2, (xy)^2, (st)^2, (yt)^2, (xyt)^3, (sty)^3, (xyst)^2 \rangle \cong D_6 \times D_6  .\]
By this analysis, the map is unique up to duality and twinness.
The above group presentation clearly indicates that this edge-biregular map is isomorphic to its twin, and hence is also a fully regular map. Meanwhile the face length $\ell$ being 6 is clearer to see when the additional (consistent) relator $(sx)^3$ is incorporated into the presentation .
\end{proof}

The only prime left to be considered in this part is $p=2$, and by exploring this case we will also establish validity of Theorem \ref{thm:main-even}.
\medskip

\noindent{\bf Proof of Theorem \ref{thm:main-even}.} \ 

We now assume $p=2$. We recall from the initial workings of this section that the possible types of map which can occur when $\nu \in \{4,8,12\}$ giving $|H|=8$, $16$ and $24$, respectively, for types $(8,8)$, $(4,8)$ and $(4,6)$. We will also need to address the only remaining case for even $k, \ell$ such that $4 \, | \, |H| = 2 \nu$, namely when $\nu = 6$ and $|H| = 12$ which occurs when the type $(k,\ell)$ is $(4,12)$ or $(6,6)$. We proceed according to the order of the group $H$.

The first possibility is dealt with quickly, because in an edge-biregular map of type $(8,8)$ with $|H|=8$ we {\em must} have $H\cong D_8$ (the stabiliser of a (single) vertex, say). The group $D_8$ contains exactly four non-central involutions, and two subgroups isomorphic to $V_4$, while the central element cannot be equal to any one of the distinct canonical involutory generators. It may be checked that up to isomorphism (since the resulting map is both self-dual and fully regular) there is just one way to present $H$ canonically, necessarily equivalent to the form given by $H_{2,1}$ in Theorem \ref{thm:main-even}.
\smallskip

When we consider $\nu=6$ for the type $(4,12)$ we clearly have the group $H = \langle x,s \rangle \cong D_{12}$, and up to twinness and duality there is only one canonical presentation. This is shown as $H_{2,2}$ in Theorem \ref{thm:main-even}.

In the case of type $(6,6)$, we have $\nu=6$ and the group is again $H \cong D_{12}$. This dihedral group contains two copies of $D_6$, one of which must be $\langle s,x \rangle$ and the other $\langle t,y \rangle$. There is only one cyclic group of order three contained in $D_{12}$ and so there are only two possibilities for the elements of order $3$, namely $sx=yt$ (which results in contradictions) or $sx=ty$ which yields the presentation $H_{2,3}$ in Theorem \ref{thm:main-even}.

\smallskip
The next possibility we look at is type $(4,8)$, with $|H|=16=2\times 8$. We know $\langle s,x \rangle = D_\ell \cong D_{8}$ is normal in $H$ since it has index two. First note that $H/D_\ell = \langle yD_\ell, tD_\ell \rangle$ has order 2 so either $y \in D_\ell$, or $t \in D_\ell$, or $yD_\ell = tD_\ell \neq D_\ell$ and so the product $ty$ is in $D_\ell$. We deal with the latter case after addressing the first two together, since they are equivalent up to twinness. Also $\langle sx \rangle \cong C_4$, being characteristic in the dihedral group, is normal in $H$. Conjugation by the elements $x$ and $s$ clearly invert $sx$ while each of $y$ and $t$ must either fix or invert $sx$, resulting in differing canonical presentations for $H$.

\smallskip

We first suppose, by our choice of the labeling of orbits, that is up to twinness, that $y \in D_\ell$.
If $y$ inverts $sx$ then we have $y \in \{ sxs, xsx \}$ and so, by the distinctness of generators, $y=sxs$, and the order of $ys$ is four. If $t$ also inverts $sx$ then $tsxt=xs$. But then $y=stsxt=txt$ which implies $y=x$, a contradiction. However, if $t$ fixes $sx$, which happens if and only if $t$ fixes $x$, we obtain
\[ \langle \; x, y, s, t \; |  \; x^2, y^2, s^2, t^2, (xy)^2, (st)^2, (ty)^{2}, (sx)^{4}, (ys)^4, (tx)^2,
ysxs \rangle \] which is the presentation of $H_{2,4}$ in Theorem \ref{thm:main-even}. Notice that this map must be supported by an orientable surface since there are no odd length relators in this presentation.
\smallskip

If, on the other hand, $y$ fixes $sx$ then $y=(sx)^2$, the central element in $D_\ell$, and hence $y(sx)^2$ is a relator, forcing the supporting surface to be non-orientable. Another consequence of $y=(sx)^2$ is that $(ys)^2$ may also appear as a relator. If $t$ also fixes $sx$, we may arrive at the presentation
\[\langle \; x, y, s, t \; |  \; x^2, y^2, s^2, t^2, (xy)^2, (st)^2, (ty)^{2}, (sx)^{4}, (ys)^2, (tx)^2, 
y(sx)^2 \rangle \]
which is $H_{2,5}$ in Theorem \ref{thm:main-even}. Otherwise, $t$ inverts $sx$, giving
\[\langle \; x, y, s, t \; |  \; x^2, y^2, s^2, t^2, (xy)^2, (st)^2, (ty)^{2}, (sx)^{4}, (ys)^2, (stx)^2, 
y(sx)^2 \rangle \]
which is the presentation of the group $H_{2,6}$ in Theorem \ref{thm:main-even}.

Now let us suppose that $ty \in D_\ell$. As before, we consider when $ty$ fixes $sx$, that is in this case $ty = (sx)^2$, the central element in $D_\ell$. If one (and hence both) of $y$ and $t$ invert $sx$ then $tsxt=xs$ so $txt=sxs$. This leads to $ty=sxsx=txtx$ which implies $y=xtx$ so $y=t$, a contradiction to the distinctness of the canonical involutions. In the other case both $y$ and $t$ fix $sx$. This is equivalent to $y$ commuting with $s$, and $t$ commuting with $x$, and this gives us the following presentation which yields a group of the right order.
\[\langle \; x, y, s, t \; |  \; x^2, y^2, s^2, t^2, (xy)^2, (st)^2, (ty)^{2}, (sx)^{4},
(ys)^2, (tx)^2, ty(sx)^2 \rangle \]
This is listed as $H_{2.7}$ in Theorem \ref{thm:main-even}. You may notice this presentation has a redundant relator, which serves to make clear the underlying full regularity of the corresponding map.

On the other hand consider when $ty$ inverts $sx$, and hence exactly one of $y$ or $t$ must fix $sx$. Up to twinness we may assume $y$ fixes $sx$ in which case $y$ is central in $H$ and $(ys)^2$ is a relator. In this case then $t$ inverts $sx$ and hence $(tsx)^2$, or equivalently $(stx)^2$, is also a relator.

For $ty \in D_\ell$ to invert $sx$ we must have $ty \in \{sxs,x,xsx,s \}$. The first two options yield a contradiciton, as we shall now see.
If $ty=sxs$ then $1=tysxs=ysys$ so $ytysx=sy$, that is $tsx=sy$ and hence $tx=y$. But then $ty=x=sxs$ which means $xs=sx$, contradicting the order of $sx$.
If $ty=x$ then $1=tyx=(tsx)^2$ so $y=sxts$, that is $x=yt=sxs$, the same contradiction.

It can be checked that the remaining options, namely $ty$ is either $xsx$ or $s$, avoid the contradictions which cause the order of the group $H$ to become too small, and they yield these two remaining canonical presentations.
\[\langle \; x, y, s, t \; |  \; x^2, y^2, s^2, t^2, (xy)^2, (st)^2, (ty)^{2}, (sx)^{4}, (ys)^2, (stx)^2, 
tyxsx \rangle \]
\[\langle \; x, y, s, t \; |  \; x^2, y^2, s^2, t^2, (xy)^2, (st)^2, (ty)^{2}, (sx)^{4}, (ys)^2, (stx)^2, 
tys \rangle \]
These are listed respectively as $H_{2.8}$ and $H_{2.9}$ in Theorem \ref{thm:main-even}. Careful comparisons of the relators in the above presentations shows that the maps are pairwise non-isomorphic.

\smallskip

The last possibility to address is the type $(4,6)$ for a group $H$ of order $3\times 8=24$. Suppose the Sylow $3$-subgroup is not normal in $H$. Then $H \in \{ SL(2,3), S_4, A_4\times C_2\}$. Now, we know $H$ has more than one involution, so it cannot be isomorphic to $SL(2,3)$. Also, $A_4$ only contains $3$ involutions, all in the subgroup which is isomorphic to $V_4$, but $H$ is generated by involutions, $H$ so cannot be isomorphic to $A_4 \times C_2$. Thus, if the Sylow $3$-subgroup is not normal in $H$, then $H\cong S_4$.
\smallskip

We will now refer to the standard permutation representation of $S_4$ on the set $\{1,2,3,4\}$ and remember the sets $T_1$ -- $T_4$ introduced in the first part of the proof of Proposition \ref{prop:three} for $p=3$. Since $\langle s,x \rangle = D_\ell \cong D_6$, we may assume, up to choice of notation and hence without loss of generality, that $x$ and $s$ are $(12)$ and $(23)$, respectively. The condition that $\langle y,t \rangle \cong V_4$ means that $y = (12)(34)$, and $t = (14)(23)$, arriving at the non-orientable map given by the presentation listed as $H_{2,10}$ in Theorem \ref{thm:main-even}:
\[\langle \; x, y, s, t \; |  \; x^2, y^2, s^2, t^2, (xy)^2, (st)^2, (ty)^{2}, (sx)^{3}, t(sy)^2, y(xt)^2 \rangle \cong S_4 \]

Now suppose the Sylow $3$-subgroup $\langle sx \rangle \cong C_3$ is normal in $H$. Then conjugation by each of $y$ and $t$ must either fix or invert $sx$.
\smallskip

Suppose $y$ inverts $sx$. Then $ysxy=xs$ so $(ysx)^2$ is a relator. If $t$ also inverts $sx$ then $tsxt=xs$. Then $txt=sxs$  and hence $txtx=sxsx$ but $y$ commutes with both $t$ and $x$ so $txtx=ytxtxy=ysxsxy=xsxs$ is self-inverse. But $sxsx$ has order three so this is a contradiction. If, on the other hand, $t$ fixes $sx$ then we obtain the presentation of the group $H_{2,11}$ from Theorem \ref{thm:main-even}:
\[ \langle \; x, y, s, t \; |  \; x^2, y^2, s^2, t^2, (xy)^2, (st)^2, (ty)^{2}, (sx)^{3}, (ysx)^2, (tx)^2
\rangle .\]

Suppose finally that $y$ fixes $sx$, in which case $ysxy=sx$ so $(ys)^2$ is a relator. In the case where $t$ also fixes $sx$, this leads to the presentation of the group $H_{2,12}$ from Theorem \ref{thm:main-even}:
\[\langle \; x, y, s, t \; |  \; x^2, y^2, s^2, t^2, (xy)^2, (st)^2, (ty)^{2}, (sx)^{3}, (ys)^2,
(tx)^2 \rangle \]
Otherwise, $t$ inverts $sx$, giving the presentation
\[ \langle \; x, y, s, t \; |  \; x^2, y^2, s^2, t^2, (xy)^2, (st)^2, (ty)^{2}, (sx)^{3}, (ys)^2, (txs)^2
\rangle \]
which represents the twin map of the edge-biregular map determined by the group $H_{2,11}$.

\smallskip

It remains to address isomorphism type, full regularity and orientability of the twelve edge-biregular maps identified above. Obviously, $H_{2,1}\cong D_8$, $H_{2,2} \cong D_{12} \cong H_{2,3}$ and $H_{2,7}\cong S_4$. Observe that the element $t \notin \langle s,x \rangle$ is central in $H_{2,4}$ and $H_{2,5}$ while $st \notin \langle s,x \rangle$ is central in $H_{2,6}$, and $y \notin \langle s,x \rangle$ is central in $H_{2,7}$, $H_{2,8}$ and $H_{2,9}$ and so the six groups are all isomorphic to $D_8\times C_2$. Further, it follows from the derivation of the remaining two groups that $H_{2,11}= \langle s,x\rangle \rtimes \langle t,y\rangle = \langle s,x\rangle \times \langle t,xy \rangle \cong D_6 \times V_4$ and $H_{2,12}= \langle s,x\rangle \times \langle t,y\rangle \cong D_6\times V_4$. It is easy to check that the twelve presentations are correct and complete (describing exactly the groups listed, albeit including some redundant relators). The maps defined by $H_{2,1}$, $H_{2,3}$, $H_{2,4}$, $H_{2,7}$, $H_{2,11}$ and $H_{2,12}$ are orientable because all their defining relations have even length in the generators $x,y,s,t$, which does not apply to the remaining six maps. Finally, the maps associated with $H_{2,1}$, $H_{2,3}$, $H_{2,7}$, $H_{2,10}$, and $H_{2,12}$ are the only fully regular ones out of the above twelve, as their groups admit an automorphism swapping $x$ with $s$ and $y$ with $t$. The proof of Theorem \ref{thm:main-even} is now complete. \hfill $\Box$
\smallskip

\smallskip

\section{The case when $p$ does not divide the order of $H$}\label{sec:nodiv}

Let $M=(H;x,y,s,t)$ be a (finite) edge-biregular map of type $(k,\ell)$ on a surface of Euler characteristic $-p$ for some prime $p$; in view of Theorem \ref{thm:main-even} we will assume that $p$ is odd and hence the surface is non-orientable. Recall that the group $H$ is assumed to be presented as in \eqref{eq:H}, and its order together with the type of the map and the characteristic of the surface are tied by the equation \eqref{eq:Eu}. We begin with an auxiliary result and omit a proof since it is almost verbatim the same as the proof of Lemma 3.2 of \cite{CPS}.

\begin{lem}\label{lem:CPS} If $q$ is a prime divisor of $|H|$ relatively prime to the Euler characteristic, then Sylow $q$-subgroups of $H$ are cyclic if $q$ is any odd prime, and dihedral if $q=2$. \hfill $\Box$
\end{lem}

From this point on until the end of this section we will assume that $p$ is {\em not} a divisor of the order of $H$. Since we are working up to duality, instead of $k\le \ell$ we henceforth assume that {\sl the $2$-part of $k$ is not smaller than the $2$-part of $\ell$}. Remember also that $H$ contains a subgroup isomorphic to $V_4$ and so $4$ divides $|H|$. Comparing this condition with equation (\ref{eq:Eu}) allows us to set $k=4 \kappa$ and $\ell=2 \lambda$ for integers $\kappa$ and $\lambda$. The stage will be set by proving solvability of $H$.

\begin{prop}\label{prop:solv}
If $p$ is such that $p\nmid |H|$, the group $H$ is solvable.
\end{prop}

\begin{proof}
We start by proving that $|D_k \cap D_\ell| \leq 4$. The group $K = D_k \cap D_\ell$ is obviously cyclic or dihedral. Suppose for a contradiction that $|K| >4$. Then $K$ contains an element $z$ of order at least $3$, so that $z =(sx)^m = (yt)^n$ for some $m$ and $n$. Clearly, $z$ commutes with $sx$ and also with $yt$. We also have $(xy) z (xy) = yx(sx)^m xy = y(xs)^m y = y (ty)^n y = (yt)^n = z$, so $z$ commutes with $xy$ as well. Now the map is carried by a non-orientable surface and so $H = \langle xy,sx,yt \rangle$ by the theory explained in section \ref{sec:alg}. Hence $z$ commutes with all the canonical generators, and thus it is central in $H$. Specifically, $z$ is central in $D_k \leq H$. It is well-known that the centre of a (non-Abelian) dihedral group is either trivial or has order $2$. But, by our assumption, the order of $\langle z \rangle$ is greater than $2$, a contradiction.
\smallskip

Next we prove that the group $H$ is a product of $D_k$ and $D_\ell$. By \eqref{eq:Eu} the assumption $p\nmid |H|$ implies that $p$ must be absorbed by the denominator of $\nu(k,\ell)$, that is, $k\ell-2(k+\ell) = rp$ for some integer $r$. Thus, $|H| = 2k\ell/r$, but by the first part of the proof we also have $|H| \geq |D_k| |D_\ell|/|D_k \cap D_\ell| \geq k\ell/4$, implying that $r\leq 8$. Using $k=4\kappa$ and $\ell=2 \lambda$ as agreed before, one has $rp = kl-2(k+l) = 8\kappa \lambda - 8\kappa - 4 \lambda = 4(2\kappa \lambda- 2\kappa - \lambda) = 4cp$. Notice that this, along with the assumption $p\nmid |H|$, also implies ${\rm gcd}(2\kappa, \lambda) = c \leq 2$. Hence $|H| = k\ell/(2c)$ where $c = {\rm gcd}(2\kappa, \lambda) \in \{1,2\}$.
\smallskip

If $c = 1$, then $|H| = k\ell/2$ and also $\lambda$ must be odd, so that $|D_k \cap D_\ell | \leq 2$. We also have $|H| \geq |D_k| |D_\ell|/|D_k \cap D_\ell|\geq k\ell/2$ and so equality holds throughout and hence $H = D_k D_\ell$ if $c=1$. In the case when $c = 2$ we have $(2\kappa \lambda- 2\kappa - \lambda) = 2p$ where $p$ is an odd prime, so $\lambda$ must be even; also, ${\rm gcd}(\kappa, \lambda/2)=1$. Now, $c = 2$ implies $|H| = k\ell/4$ and $|D_k \cap D_\ell | \leq 4$, and as we also have $|H| \geq |D_k| |D_\ell|/|D_k \cap D_\ell| \geq k\ell/4$ we conclude that equality holds throughout and hence $H = D_k D_\ell$.
\smallskip

We may now complete the proof by invoking the result of Huppert \cite{Hup} that the product of two dihedral groups is solvable.
\end{proof}

\smallskip

The fact that $H$ is solvable yields that it has a non-trivial Fitting subgroup $F$; recall that $F$ is the largest nilpotent normal subgroup of $H$. In particular, $F$ is a direct product of its Sylow subgroups. By what we know about the Sylow subgroups of $H$ from Lemma \ref{lem:CPS} we have $F=F_1\times F_2$, where $F_1$ is cyclic, of odd order (possibly trivial), and $F_2$ (if non-trivial) is a cyclic or a dihedral $2$-group; we will henceforth split our analysis according to this dichotomy.
\smallskip

As a general remark, observe that we may assume $F\ne H$. Indeed, if $F=H$, then $F_1$ would have to be trivial (otherwise $F$ could not be generated by involutions) and $F_2$ would have to be non-cyclic (to contain enough distinct involutions), so that $H=F=F_2$ would have to be dihedral. But edge-biregular maps with dihedral groups $H$ have already been classified in \cite{Jea}. Without giving details we just state that, as a consequence of the table displaying the classification results in \cite{Jea}, the only edge-biregular maps of Euler characteristic $-p$ for an odd prime $p$ determined by a dihedral group of automorphisms are the first two maps in Theorem \ref{thm:main-odd} defined by the groups $H_{p(1)}$ and $H_{p(2)}$.

\subsection{The case when the Fitting subgroup is cyclic}\label{sub:cyc}
\medskip

From now on we will {\sl assume that $F$ is cyclic}. In such a case $F$ contains either no involution (if $|F|$ is odd) or a unique involution (if $F_2$ is non-trivial cyclic $2$-group). This implies that $F$ can contain at most one of the four involutions $s,t,x,y$.
\smallskip

We will show that $F$ cannot have index $2$ in $H$. Indeed, suppose $[H:F]=2$. If $F$ contains no generating involution from the set $\{s,t,x,y\}$, then the index-2 condition implies that $sx,xy,yt\in F$. But since $H=\langle sx,xy,yt\rangle$ we would have $F=H$, a contradiction. Thus, let one of $\{s,t,x,y\}$ be the unique generating involution of $H$ contained in $F$. We may without loss of generality assume that this element is $y$, as all the other cases are handled by symmetries in the forthcoming argument. Now $\langle y\rangle$ is characteristic in $F$ and therefore normal in $H$, and so $y$, being an involution, is also central in $H$. Further, from $s,t\notin F$ we have $st\in F$, and by uniqueness of the involution in $F$ it follows that $y=st$.
\smallskip

As now $s,x\notin F$, we have $sx\in F$ and so $F$ contains the cyclic group $\langle sx\rangle$. If its order is odd, then $F$ also contains the cyclic group $K=\langle sx\rangle\langle y\rangle$. But then, recalling centrality of $y$ in $H$, conjugation by, say, $s$ inverts every element of $K$ and so $L=K\rtimes \langle s\rangle$ is a dihedral group. From $s,xs,y\in L$ and $y=st$ we also have $t,x\in L$ and so $L=H$ is dihedral, the case we have already disposed of. If the order of $\langle sx\rangle$ is even, then $y=(sx)^j$ for $j$ equal to half of the order of $sx$, since both elements are an involution in the cyclic group $F$. But we saw earlier that $y=st$, giving $st = (sx)^j$ and hence $t=s(sx)^j$. This, however, with $y=(sx)^j$ shows that $H=\langle s,x\rangle$, which again means that $H$ is dihedral.
\smallskip

Altogether, we have shown that for non-dihedral $H$ and for a cyclic $F$ we must have $[H:F] > 2$.
Now, by the available theory $H/F$ embeds in ${\rm Aut}(F)$, and as the latter is Abelian if $F$ is cyclic, we conclude that $H/F$ is Abelian. But $H/F$ is generated by four elements of order at most $2$, so that $H/F\cong C_2^m$ for some $m$ such that $2\le m\le 4$. Further, since $H=\langle st,sx,xy\rangle = \langle st,ty,xy\rangle$ it follows that $m\in \{2,3\}$ and both $sxF$ and $tyF$ have order at most $2$, so that $(sx)^2,(ty)^2\in F$.
\smallskip

From earlier calculations we recall that both $k,\ell$ are even and greater than $2$, and assuming that the $2$-part of $k$ is not smaller than the $2$-part of $\ell$ we have the following: {\sl If $c={\rm gcd} (k/2, \ell/2)$, then $c\in \{1,2\}$, $k$ is a multiple of $4$, and $|H|=k\ell/(2c)$.} In particular, note that $8\nmid \ell$, and $c=1$ or $2$ according as $\ell/2$ is odd or even.
\smallskip

Under these conditions we first show that $ty\notin F$. Indeed, suppose that $ty\in F$. Observe that the order of $(sx)^c$ is $\ell/(2c)$ and hence odd, so that $(sx)^2\in F$ implies $(sx)^c\in F$ also for $c=1$. As $ty,(sx)^c\in F$ and the orders of the two elements, $k/2$ and $\ell/(2c)$, are relatively prime, it follows that $(k/2)(\ell/(2c)) \le |F| = |H|/[H:F] \le k\ell/(2c[H:F])$, which yields $[H:F]\le 2$, a contradiction.
\smallskip

It follows that $ty\notin F$. But we know that $(ty)^2\in F$, and we may use essentially the same chain of inequalities as above, with $k/2$ replaced by $k/4$ (which is the order of $(ty)^2$), to conclude that $[H:F]\le 4$. We saw, however, that $[H:F]$ is a power of $2$ and greater than $2$, so that $[H:F]=4$, and we must have equalities in the above chain throughout. In more detail, and using the fact that $F$ is cyclic, we have $F=\langle (sx)^c\rangle\langle(ty)^2\rangle = \langle (sx)^c(ty)^2 \rangle\cong C_n$ for $n=k\ell/(8c)$, with $ty\notin F$. Observe that the order of $(sx)^2$ is odd in both cases for $c\in\{1,2\}$, and is and relatively prime to $k/4$ (the order of $(ty)^2$).
\smallskip

We show that $F$ contains none of the generating involutions $s,t,x,y$ of $H$, and at most one of the involutions $st$, $xy$. For if $F$ contained one of $t,y$, then this element would have to coincide with the central involution $(ty)^{k/4}$, but such an equality quickly gives $t=1$ or $y=1$. If $s\in F$ (and the case $x\in F$ is done similarly), then $s$ would commute with $(sx)^2$, which is equivalent to $(sx)^4=1$. Then, $\ell/2$ would have to divide $4$ and since $8\nmid \ell$ we would have $\ell=4$. But this would give $H=k\ell/4=k$, so that $H$ would be dihedral.

To address the remaining part, note that by non-orientability we know that $H=\langle st,xy,yt\rangle$. From $H/F=\langle stF,xyF,ytF\rangle$ and $yt\notin F$ while $(yt)^2\in F$, together with $[H:F]=4$, it follows that at most one of $st$, $xy$ can be contained in $F$.
\smallskip

In what follows we will without loss of generality assume that $st\notin F$. As $F$ and $\langle s,t\rangle$ now intersect trivially, with the help of the above finding this means that the semi-direct product
\begin{equation}\label{eq:F:V4st}
F\rtimes \langle s,t\rangle = \langle (sx)^c(ty)^2 \rangle \rtimes \langle s,t\rangle \cong C_n\rtimes V_4
\end{equation}
for $n=k\ell/(8c)$ has order $4n=|H|$ and so $H=F\rtimes \langle s,t\rangle \cong C_n\rtimes V_4$. If $n=|F|$ is even, then the unique non-trivial involution $(ty)^4\in F$ generates a subgroup isomorphic to $C_2$ that is characteristic in $F$ and hence normal in $H$, which means that $(ty)^4$ is central in $H$. But then $H$ would contain the subgroup $\langle(ty)^{k/4} \rangle \times \langle s,t\rangle \cong C_2^3$, contrary to the fact that $H$ has dihedral Sylow $2$-subgroups. It follows that $n$ is odd and so is $\kappa=k/4$ (and we know the same about $\ell/(2c)=\lambda/c$); also, both $xy$ and $st$ must lie outside $F$, and the Sylow $2$-subgroups of $H$ are isomorphic to $V_4$.
\smallskip

In the proof of solvability of $H$ we have encountered the equation $2\kappa\lambda - 2\kappa -\lambda = cp$ for some $c\in\{1,2\}$. If $c=2$, then $\lambda$ is exactly divisible by $2$, and then oddness of $\kappa$ with $2\kappa(\lambda-1) = 2p +\lambda$ gives a contradiction as the right-hand side is divisible by $4$ while the left-hand side is not. It follows that $c=1$, and so $sx\in F$; observe that then $sy\notin F$ as in the opposite case we would have $(xs)(sy)=xy\in F$ which has already been excluded.
\smallskip

We now let $u=sx$ and $v=ty$; note that our cyclic group $F$, of order a product of two odd and relatively prime numbers $\ell/2$ and $k/4$ (the orders of $u$ and $v^2$), is generated e.g. by $uv^2$. We saw that $y,ys,yt\notin F$, so that by \eqref{eq:F:V4st} for $c=1$ we must have $y=wst$ for some $w\in F$. The fact that $w=yts$ commutes with $u=sx$ is equivalent to $ty(sx)yt=xs= u^{-1}$, so that conjugation by $v$ inverts $u$. Similarly, $w=yts$ commutes with $v^2\in F$, which translates to $[s,v^2]=1$, and as $u=sx$ commutes with $v^2$ we also have $[x,v^2]=1$.
\smallskip

In somewhat more detail, let $w=u^a(v^2)^b=u^av^{2b}$ for uniquely determined integers $a,b$ such that $0\le a< \ell/2$ and $0\le b< k/4$. Using the facts that $s$ inverts $u$ and commutes with $v^2$ and $[u,v^2]=1$, from $y=wst$ it follows that $v^{-2}=(yt)^2= (u^av^{2b}s)^2 = v^{4b}$, so that $v^{4b+2}=1$, and for $b$ in the above range we have $4b+2=k/2$ (the order of $v$) and so $b=(k-4)/8$. Normality of $\langle u\rangle$ in $H$ (being characteristic in $F$) further implies that $yuy=u^j$ for for some $j$, $1\le j< \ell/2$, such that $j^2 \equiv 1$ mod $\ell/2$. Since $v$ inverts $u$, we obtain $u^{-1}=tyuyt = tu^jt$, which implies $tut=u^{-j}$. Observing now that conjugation by $st$ maps $u^a$ onto $u^{aj}$ and inverts $v^2$, from $y=u^av^{2b}st$ we obtain $1=(u^av^{2b}st)^2=u^{a(j+1)}$, so that $a(j+1)\equiv 0$ mod $\ell/2$.
\smallskip

Going one step further and using properties derived above, from $y=u^av^{2b}st$ we have $xy = xu^asv^{2b}t = u^{-a-1}v^{2b}t$, and as conjugation by $t$ inverts $v^2$ and maps $u$ onto $u^{-j}$ one obtains $1=(u^{-a-1}v^{2b}t)^2 = u^{(a+1)(j-1)}$, which gives $(a+1)(j-1)\equiv 0$ mod $\ell/2$. Subtracting the last congruence from $a(j+1)\equiv 0$ mod $\ell/2$ yields $2a+1\equiv j$ mod $\ell/2$. As $2^{-1}=(\ell+2)/4$ mod $\ell/2$ (which is odd), it follows that $a \equiv (j-1)2^{-1}=(j-1)(\ell+2)/4$ mod $\ell/2$, giving a unique value of $a\in \{0,1,\ldots, \ell/2-1 \}$. Observe also that for this value of $a$ and for any $j$ such that $j^2\equiv 1$ mod $\ell/2$ one has $(j+1)a = (j^2-1)2^{-1} \equiv 0$ mod $\ell/2$, which is the congruence obtained earlier.
\smallskip

The last step will be reintroducing the notation $k=4\kappa$ and $\ell=2\lambda$ for odd and relatively prime $\kappa$ and $\lambda$, and observing that $2b+1 = \kappa$ and so $1=tyv^{2b}u^as =v^\kappa u^as$. The last relation is equivalent to $v^{\kappa-1}u^a=v^{-1}s$, and from $[u,v^2]=1$ (a consequence of $[s,v^2]=1=[x,v^2]$) and oddness of $\kappa$ it follows that $1=[u,v^{\kappa-1}u^a] = [u,v^{-1}s]$ and commutation of $u$ and $v^{-1}s$ is equivalent to $u$ being inverted by conjugation by $v$. Summing up, the above facts well-define a group $H=H_{p,j}$ of order $k\ell/2$ generated by four involutions $s,t,x,y$ and presented as follows, with $u=sx$ and $v=ty$:
\begin{equation}\label{eq:Fcyc}
H_{p,j} = \langle \, x,y,s,t\, | \, x^2,\,y^2,\,s^2,\,t^2,\,(xy)^2,\,(st)^2,\,u^{\lambda},\,
v^{2\kappa}, \, [s,v^2],\,[x,v^2],\, tutu^j,\, v^\kappa u^a s\, \rangle
\end{equation}
for a non-negative integer $j < \ell/2$ such that $j^2\equiv 1$ mod $\lambda$, with $a=(j-1)(\lambda+1)/2$. This is the presentation appearing as a third item in Theorem \ref{thm:main-odd}. Note that here $F=\langle u,v^2\rangle = \langle uv^2\rangle$ is cyclic, of order $\kappa\lambda = k\ell/8$, and $H_{p,j}=F\rtimes \langle s,t\rangle$, as derived earlier.
\medskip

\noindent {\bf Proof of correctness of the presentation \eqref{eq:Fcyc}}
\medskip

Let $G_0=\langle s,x\, |\, s^2,x^2,(sx)^{\lambda}\rangle \cong D_\ell$ and let $u=sx$. Let us introduce a pair of automorphisms $\theta$ and $\tau$ of $G_0$ completely defined by letting $\theta(s)=s$ and $\theta(u)=u^{-j}$ for some $j$ such that $j^2\equiv 1$ mod  $\lambda$, and $\tau(u)=u^j$, $\tau(x)=x$. These definitions imply, for example, that $\theta(x)=su^{-j}$ and $\tau(s)=u^jx$. It can be easily verified that $\theta$ and $\tau$ commute and both are of order two. It follows that we have a well-defined split extension of $G_0$ by the subgroup $V_4\cong \langle \theta,\tau\rangle < {\rm Aut}(G_0)$; the order of the extension is $4\ell$. By general knowledge on split extensions the new group has an equivalent representation in the form $G_1=\langle s,x\rangle\rtimes \langle t,y\rangle$, where $t,y$ are two commuting involutions acting on $G_0$ by conjugation the same way as the two earlier automorphisms do, that is, $\theta(z)=tzt$ and $\tau(z)=yzy$ for every $z\in G_0$. Note that this also implies that the subgroups $G_0$ and $\langle t,y\rangle$ intersect trivially. Further, using $u=sx$ and $v=ty$ it can be verified that in $G_1$ one has the relations $uvu=v$; moreover, we have $[s,t]=[x,y]=1$ by the definition of the two automorphisms and their conversion to conjugations. It follows that the group $G_1$ has a presentation of the form
\begin{equation}\label{eq:aux1}
G_1=\langle x,y,s,t\, | \, x^2,\,y^2,\,s^2,\,t^2,\,(xy)^2,\,(st)^2,\,u^\lambda,\,v^2,\,   tutu^j,\,uvuv^{-1}\rangle
\end{equation}
where, again, $u=sx$, $v=ty$, and $j$ is an integer such that $j^2\equiv 1$ mod $\lambda$.
\smallskip

Next, let us consider the group $G_2$ generated by the same involutions as $G_1$ but with a presentation obtained from that of $G_1$ by omitting the relator $v^2$ and adding the conditions of $v^2$ commuting with $s$ and $x$:
\begin{equation}\label{eq:aux2}
G_2=\langle x,y,s,t\, | \, x^2,\,y^2,\,s^2,\,t^2,\,(xy)^2,\,(st)^2,\,uvuv^{-1},\, u^\lambda,\,tutu^j,\,[s,v^2],\,[x,v^2]\rangle
\end{equation}
The relators $uvuv^{-1}$ and $tutu^j$ imply $yuy=u^j$ and $txt=su^{-j}$, and these together with $tut=u^{-j}$ and $yxy=x$ show that $\langle u,x\rangle = \langle s,x\rangle$ is a normal subgroup of $G_2$. By inspection, in the absence of any condition involving the element $v$ the presentation of $G_2/\langle s,x\rangle$ reduces to $\langle t,y\, |\, t^2,y^2\rangle$, which is an infinite dihedral group; hence $G_2$ is infinite. The subgroup $N=\langle v^2\rangle$ is obviously normal in $G_2$ and as $G_2/N\cong G_1$ we conclude that $N$ is isomorphic to an infinite cyclic group. The subgroup $\kappa N = \{(v^{2\kappa})^i;\ i\in \mathbb{Z}\}$ is characteristic in $N$ and so normal in $G_2$.  Applying the Third isomorphism theorem we obtain $G_2/N \cong (G_2/\kappa N)/(N/\kappa N)$ and so $|G_2/\kappa N|=\kappa |G_2/N|$. Since $G_2/N \cong G_1$, the new group $G_3=G_2/\kappa N$ of order $8\kappa\lambda$ has a presentation obtained from that of \eqref{eq:aux2} by adding the relator $v^{2\kappa}$:
\begin{equation}\label{eq:aux3}
G_3=\langle x,y,s,t\, | \, x^2,\,y^2,\,s^2,\,t^2,\,(xy)^2,\,(st)^2,\,uvuv^{-1},\, u^\lambda,\,tutu^j,\,[s,v^2],\,[x,v^2],\,v^{2\kappa}\rangle
\end{equation}

The last step is to consider the element $z=v^\kappa u^a s\in G_3$, where $a=2^{-1}(j-1)$ mod $\lambda$ as before. For the calculations that follow it is useful to observe that from the earlier congruence $a(j+1)\equiv 0$ mod $\lambda$ we have $aj\equiv -a$ mod $\lambda$, so that $tu^at=u^a$ and $yu^ay=u^{-a}$; note also that $u^{2a+1}=u^j$. We begin by showing that $z$ is an involution. Indeed, using $[u,v^2]=1$, the facts that both $s$ and $v$ invert $u$, and the properties of $u$ listed before one obtains $(u^asv^\kappa)^2 = u^{2a}sv^{-1}v^{\kappa+1}s v^\kappa = u^\lambda u^{-1}sv^{-1}sv = u^\lambda yu^{-1}y = 1$, which yields $z^2=1$.
\smallskip

We show even more; namely, that $z$ is a central involution of $G_3$. To show that $z$ commutes with $s$, note that the above implies that $zs=v^\kappa u^a$ is an involution and so  $(sz)^2= s(v^\kappa u^a)^2s=1$, so that $[s,z]=1$. With the help of the fact that conjugation by $t$ preserves $u^a$ and $v^\kappa$ (note that $v^\kappa = v^{-\kappa}$) we obtain $[z,t]=v^\kappa u^a st\cdot tv^\kappa u^as= z^2=1$. As $\kappa$ is odd and so $v^\kappa$ inverts $u$, it follows that  $[x,z]=xv^\kappa u^a(sx)v^\kappa u^a s = xv^\kappa u^{a+1}v^\kappa u^{a+1}x =1$. Last, using inversion of $u^a$ and preservation of $v^\kappa$ by conjugation by $y$, as well as preservation of $u$ by conjugation by $sv$, one has $[z,y] =v^\kappa u^as v^{\kappa}u^{-a}ysy = v^\kappa u^a v^{\kappa-1} u^{-a}svysy = v^{-1}svysy = 1$.
\smallskip

We have proved that for $z=v^\kappa u^a s$ the subgroup $\langle z\rangle\cong C_2$ is central in $G_3$. But this means that the group $G_3/\langle z\rangle$, which is isomorphic to $H_{p,j}$, has order $|G_3|/2= 4\kappa\lambda = k\ell/2$, as claimed. This completes the proof of correctness of the presentation \eqref{eq:Fcyc}; the relator $uvuv^{-1}$ can be omitted since it is a consequence of the remaining ones, as shown in the paragraph immediately preceding the presentation \eqref{eq:Fcyc}.
\medskip

For completeness we show that none of the edge-biregular maps with automorphism groups $H_{p,j}$ presented as in \eqref{eq:Fcyc} are fully regular. Indeed, in the opposite case $H_{p,j}$ would have to admit an automorphism interchanging $s$ with $x$ and $t$ with $y$. In such a case, however, along with $s=v^\kappa u^a$ the group $H_{p,j}$ would also have to admit the relation $x=v^{-\kappa}u^{-a}= u^av^\kappa$. This would imply $xs=u^{2a}$ and hence $u^{2a+1}=1$, and as $2a+1\equiv j$ mod $\lambda$ we would have $u^j=1$, contrary to $j^2\equiv 1$ mod $\lambda$.
\smallskip

Finally, the Euler-Poincar\'{e} formula yields that $p=2\kappa\lambda -2\kappa - \lambda$ so $p+1 =(2\kappa-1)(\lambda - 1)$. We know $\kappa$ and $\lambda$ are both odd so letting $p+1 = 2^\alpha b d'$, where $b \equiv 1$ mod $4$, and $d'$ is odd, then we have $2\kappa - 1 = b$ and $\lambda - 1 = 2^\alpha d'=d$, which, so long as $\kappa$ and $\lambda$ are coprime, will yield an edge-biregular map as described above.
Hence, for every factorisation $p+1=bd$ such that $b\equiv 1$ mod $4$ and ${\rm gcd}(b+1,d+1)=1$ we have such a map of type $(2(b+1), 2(d+1))$, completing the analysis related to the third item of Theorem \ref{thm:main-odd}.

\subsection{The case when the $2$-part of the Fitting subgroup is dihedral}
\medskip

Recall that we are investigating an edge-biregular map $M=(H;x,y,s,t)$ of type $(k,\ell)$ with both entries even, on a surface of Euler characteristic $-p$ for some odd prime $p$; at the beginning of section \ref{sec:nodiv} we also made the assumption that the $2$-part of $k$ is not smaller than the $2$-part of $\ell$ and we had $k=4\kappa$ and $\ell=2\lambda$. By Proposition \ref{prop:solv} we know that $H$ is a solvable group and so has a non-trivial Fitting subgroup $F$, of which we may assume that $F\ne H$. By earlier results and observations we also know that either $F$ is cyclic, or $F$ has a dihedral $2$-part, denoted $F_2$. We have dealt with the first possibility in subsection \ref{sub:cyc} and {\sl from now on we will assume that $F_2$ is dihedral}, which of course means that $|F_2| \geq 4$. Our next result places a substantial restriction on $F_2$ and $D_k$ (the vertex-stabilizer in $M$).

\begin{prop}\label{prop:F2}
If $F_2$ is dihedral, then $D_k$ is a $2$-group, $[D_k:F_2]=2$, and $k$ is a multiple of $8$ while $\ell/{(2c)}$ is odd.
\end{prop}

\begin{proof}
We have assumed that the $2$-part of $k$ is not smaller than that of $\ell$, and we know that $|H|= k\ell/(2c)$ for $c=1,2$. Analysis of equation (\ref{eq:Eu}) with these conditions yields that a Sylow $2$-subgroup of $H$ is contained in $D_k$. As $F_2$ is the Sylow $2$-subgroup of $F$, normality of $F$ in $H$ implies that $F_2$ is a subgroup of $D_k$.

\smallskip

Suppose $F_2 = D_k = \langle t,y\rangle$. Then $H/{F_2}$ is generated by $sF_2$ and $xF_2$, each of which has order less than or equal to two. But $ |H|/|F_2| = |H|/k = \ell/(2c) = \lambda/c$ and since $c = {\rm gcd}(2\kappa,\lambda)\le 2$ it follows that $\lambda/c$ must be an odd integer and hence the group $H/F_2$ does not contain any involution. Hence $s,x \in F_2$ and so $H= D_k$, contrary to our assumption that $H$ is not dihedral.
\smallskip

To show that $[D_k:F_2]\le 2$, up to the choice of labelling of orbits we may assume $F_2 = \langle y, (ty)^\mu \rangle$ for some $\mu$ [the alternative option would be $F_2 = \langle t, (ty)^\mu \rangle$]. Now, $F_2$ is normal in $H$ (because it is characteristic in $F$) and so $y^{ty} = ty y yt = tyt \in F_2$. But $y \in F_2$ also so this means $(ty)^2 \in F_2$, which implies $|D_k : F_2| \leq 2$. Hence $D_k$ is a $2$-group; the conclusions about $k$ and $\ell$ are obvious from the above.
\end{proof}

Since $F_2$ is normal in $H$ and hence also in $\langle t,y\rangle \cong D_k$, we have established that there are only three possibilities for a dihedral $F_2$ if $k\ge 8$: either $F_2=\langle t,y\rangle$, or $F_2$ is one of $\langle t,(ty)^2\rangle$, $\langle y, (ty)^2\rangle$; in particular, $k$ must be a power of $2$. In the first case we have $H/F_2$ generated by (at most) four involutions, but oddness of $\ell/(2c)=|H/F_2|$ implies that the generating involutions $s,t,x,y$ all belong to $F_2$ and hence $H=F_2\cong D_k$; such maps with a dihedral automorphism group have already been sorted out. In what follows we will assume that $F_2=\langle z, (ty)^2\rangle$ for some $z\in \{t,y\}$.
\smallskip

If $k\ge 16$, the cyclic subgroup $\langle (ty)^2\rangle$ of $F_2$, of order $k/4$, is characteristic in $F_2$ and therefore normal in $H$. Note that for $k=8$ and $F_2\cong C_2\times C_2$ this need not be valid. For now we will {\em assume} that $\langle (ty)^2\rangle$ is normal in $H$ also for $k=8$ and we will return to the opposite case later.
\smallskip

Before proceeding we make a remark about the case $c=2$. By the proof of solvability of $H$, for $c=2$ the subgroups $D_k$ and $D_\ell$ intersect in a subgroup isomorphic to $V_4$. For $k\ge 16$ the subgroup $\langle (ty)^{k/4} \rangle \cong C_2$ generated by the centre of $F_2$ is characteristic in $F_2$ and hence normal in $H$, which means that $(ty)^{k/4}$ is a central involution also in $H$. Note that this is now also valid for $k=8$ because of the assumption made in the previous paragraph. If $c=2$ we may also assume that $\ell\ge 12$ (as for $\ell=4$ we  would have $H$ a dihedral group), and so for the central element of $\langle s,x\rangle \cong D_\ell$ we have $(sx)^{\ell/4}=(ty)^{k/4}$.
\smallskip

Normality of $\langle (ty)^2\rangle$ of $F_2$ implies that $s(ty)^2s=(ty)^{2i}$ and $x(ty)^2x=(ty)^{2j}$ for some integers $i,j$ such that $i^2\equiv j^2\equiv 1$ mod $k/4$; the exponents $2$ in the congruences match the orders of $s$ and $x$. It follows that $sx(ty)^2xs = s(ty)^{2j}s = (ty)^{2ij}$. But as $(sx)^{\ell/(2c)}$ is either the identity or a central element of $H$, we must have $(ij)^{\ell/(2c)} \equiv 1$ mod $k/4$. In view of the previous two congruences, oddness of $\ell/(2c)$ implies that $1\equiv (ij)^{\ell/(2c)}\equiv ij$, so that $sx$ commutes with $(ty)^2$. Since $k$ and $\ell/(2c)$ are relatively prime and $k/4$ is even while $\ell/(2c)$ is odd, it follows that the subgroup $J=\langle sx,(ty)^2\rangle$ of $H$ is cyclic, of order $k\ell/(8c)$, and generated by the product $(sx)(ty)^2$.
\smallskip

From the fact that $sx$ commutes with $(ty)^2$ we have $s(ty)^2s=x(ty)^2x$ and this element is in $\langle (ty)^2\rangle$ by normality of this subgroup in $H$. Note that we cannot have $x\in J$; in the opposite case $x$ would have to be equal to $(ty)^{k/4}$, the unique involution in $J$, but then $x$ would commute with $sx$ and hence with $s$, giving $(sx)^2=1$, contrary to $\ell\ge 6c$. Thus, the semi-direct product $K=J\rtimes \langle x\rangle$ is a subgroup of $H$ of order $k\ell/(4c)$, and hence normal (of index $2$) in $H$.
\smallskip

The subgroup $\langle sx\rangle < J$ as the unique cyclic subgroup of $K$ of order $\ell/2$ is characteristic in $K$ and hence normal in $H$. Thus, $y(sx)y=(sx)^a$ and $t(sx)t=(sx)^b$ for some positive integers $a,b < \ell/2$, with both $a,b$ odd if $c=2$. By $[x,y]=1=[s,t]$ we then have $ysy=s(xs)^{a-1}$ and $(sy)^2= (xs)^{a-1}$; similarly, $txt=(xs)^{b-1}x$ and $(tx)^2=(xs)^{b-1}$. By normality of $F_2$ in $H$, the element $sys$ or $xtx$ (depending on whether $y\in F_2$ or $t\in F_2$) is equal to either the central element $(ty)^{k/4}$ or an involution of the form $(ty)^{2j}z$ for some $j < k/4$. In the first case, either $sys$ or $xtx$ would commute with $xs$, which is easily seen to be equivalent to $(sy)^2=1$ or $(tx)^2=1$. In the second case, either $(sy)^2$ or $(tx)^2$ are a power of  $(ty)^2$ and so their order is a power of $2$. But we have also established that $(sy)^2= (xs)^{a-1}$ or $(tx)^2=(xs)^{b-1}$. Both $(xs)^{a-1}$ and $(xs)^{b-1}$ have, however, odd order; this is obvious for $c=1$ and for $c=2$ it follows from oddness of $a$ and $b$. The two order parities can be matched only if $(sy)^2=1$ or $(tx)^2=1$. In both cases we have established that, depending on whether $y\in F_2$ or $t\in F_2$, we have $(sy)^2=1$ or $(tx)^2=1$, i.e., $[s,y]=1$ or $[t,x]=1$.
\smallskip

Next, we show that $ty\notin K$. Indeed, if $ty\in K$, then $K$ would contain the cyclic group $\langle ty\rangle \cong C_{k/2}$ and also the dihedral group $\langle s,x\rangle\cong D_\ell$. By comparing their orders with $|K|=k\ell/(4c)$ it follows that the two groups intersect non-trivially. If $c=2$ then the two groups would have to intersect in a group of order $4$, which would have to be cyclic and dihedral at the same time, a contradiction. It follows that $c=1$ and the only non-trivial element in their intersection is an involution. Since the only involution contained in the cyclic group is the central one, we would have $(ty)^{k/4}=s(sx)^j$ for some $j$. But as the central element commutes with $sx$, the same must hold for $s(sx)^j$ and this is easily seen to be equivalent to $(sx)^2=1$, contrary to the bound on $\ell$. Thus, $ty\notin K$, as claimed. But then, from $H/K= \langle tK, yK \rangle \cong C_2$ it follows that either $t\in K$ or $y\in K$. This means that $K$ contains a dihedral subgroup of order $k/2$, which is in fact the unique Sylow $2$-subgroup $F_2$ of $F$. Hence $F_2<K$ and, moreover, from $[t,x]=1$ or $[s,y]=1$ it follows that $K=\langle sx\rangle\cdot F_2\cong C_{\ell/2} \times D_{k/2}$. (In fact, at this stage it follows that $K$ is the Fitting subgroup $F$ of $H$; to see this one only needs to see that $F\ne H$ but in the opposite case $F/F_2$ would be trivial (being generated by involutions) and we would be back in the case $F=D_k$. However, in the light of the conclusion we will arrive at, the fact that $K=F$ will turn out to be irrelevant.)
\smallskip

To finish this part of our argument we explore the fact that $x\in K$. By the above structural information this means that $x\in F_2$, and so $x$ would commute with $sx$, which is equivalent to $(sx)^2=1$, contrary to our bound on $\ell$.
\smallskip

It thus remains to investigate the case when $k=8$ and the subgroup $\langle (ty)^2\rangle\cong C_2$ of $F_2=\langle z, (ty)^2\rangle\cong C_2\times C_2$ for some $z\in \{t,y\}$ is not a normal subgroup of $H$, with $|H|=k\ell/(2c) = 4\ell/c$ (note that now $8$ exactly divides the order of $H$). For definiteness we will assume that $t\in F_2$; the case when $y\in F_2$ is done in a completely analogous way by replacing $s$ with $x$ and $t$ with $y$ in the subsequent arguments.
\smallskip

Thus, let $F_2=\{ 1,t,yty, (ty)^2\}$; by our assumption $y\notin F_2$. If $x$ was in $F_2$ then clearly $x\ne 1,t$, and as $x=yty$ implies $ty$ has order two, the only possibility would be $x=(ty)^2$. This would mean that $F_2=\langle t,x\rangle$; by normality in $H$, conjugation by $s$ preserves $F_2$. We cannot have $[s,x]=1$ as this contradicts our bound on $\ell$, and since $[s,t]=1$ it follows that $sxs=xt$ and hence $(sx)^4=1$, contrary to $\ell\ge 6c$. Therefore $x\notin F$, and, as we know, $y\notin F$ either, but note that $xy\in F$. Indeed, in the opposite case, by normality of $F_2$ and its trivial intersection with $\langle x,y\rangle$, the subgroup $F_2\rtimes \langle x,y\rangle$ of $H$ would have order $16$, contrary to $8$ exactly dividing the order of $H$. A calculation as above shows that the only option for $xy\in F$ is $xy=(ty)^2$ which is is equivalent to $txty=1$ (hence $tx$ has order $4$). Observe that this relation also shows that conjugation of $F_2$ by $x$ fixes $(ty)^2$, and commutativity of $xy,t\in F_2$ implies $xtx=yty$.
\smallskip

Since $F_2$ is normal in $H$, it is preserved by conjugation by $s$, which fixes $t$. If $s$ fixed $(ty)^2$, then with $x$ fixing $(ty)^2$ the subgroup $\langle (ty)^2\rangle$ would be normal in $H$, contrary to our assumption made in this special case $k=8$. It follows that conjugation by $s$ induces an automorphism of $F_2$ fixing $t$ and transposing $yty$ with $(ty)^2$; the relation $s(yty)s=(ty)^2$ is equivalent to $(syty)^2t=1$. The composition of the two conjugations, namely, $w\mapsto (xs)w(sx)$ for $w\in F_2\cong C_2\times C_2$, induces and automorphism in ${\rm Aut}(C_2\times C_2)\cong S_3$ represented by the cycle $t\mapsto yty\mapsto (ty)^2\mapsto t $ of length $3$. In particular, conjugation by $(sx)^3$ centralizes $t$. If now $\ell/2=3q+r$ for $r=\pm 1$, then $[(sx)^3,t]=1$ implies $[(sx)^r,t]=1$, contrary to the fact established earlier that conjugation by $sx$ does not fix $t$. It follows that $r=0$ and hence $\ell/2$ is a multiple of $3$.
\smallskip

We are now approaching derivation of a presentation of $H$. Observe first that the relation $xtx=yty$ derived earlier simplifies $(syty)^2t=1$ to $(sxy)^2t=1$. From $y=txt$ and the fact that $(sx)^3$ centralizes $F_2$ we obtain $y(sx)^3y = (xs)^3$, so that $\langle (sx)^3\rangle$ is a normal cyclic subgroup of $H$. If $c=2$, then, by the findings in the previous paragraph, the odd integer $\ell/4$ is a multiple of $3$ and so $(sx)^{\ell/4}$ is a central element of $H$ as obviously commutes with $s$ and $x$, and also with $t$ (because $(sx)^3$ does) and hence also with $y=txt$. But then $\langle (sx)^{\ell/4}\rangle$ must be a subgroup of $F_2$, otherwise its product with $F_2$ would be isomorphic to $C_2^3$, contrary to the Sylow $2$-subgroups of $H$ being dihedral. We cannot have $(sx)^{\ell/4}$ equal to $t$ or $yty$ because the two elements do not commute with $y$, so that the only option is $(sx)^{\ell/4}=(ty)^2$, but while the element on the left commutes with $s$ the one on the right does not. It follows that $c=1$ and hence $|H| = 4\ell$.
\smallskip

We note that normality of $\langle (sx)^3\rangle$ cannot be extended to normality of $\langle sx\rangle$, as otherwise we would have $t(sx)t=(sx)^d$ and then $(xt)^2= (sx)^{d-1}$ for some $d$, which are elements of orders of different parity ($4$ versus some (odd) divisor of $\ell/2$). This shows, as an aside, that the Fitting subgroup of $H$ is $F\cong \langle (sx)^3\rangle\times F_2 \cong C_{\ell/6}\times V_4$. It may also be useful to note that $y=txt$ implies that $sy = stxt=t(sx)t$, showing that the orders of $sx$ and $sy$ are the same, namely, $\ell/2$.
\smallskip

This way, in the case when $k=8$ and $F_2=\langle t,yty\rangle$, with $\ell/2=3m$ for some odd $m\ge 1$, we have arrived at a presentation of $H=H_p$ of the form
\begin{equation}\label{eq:Ht}
H_p = \langle x,y,s,t\, | \, x^2,\,y^2,\,s^2,\,t^2,\,(xy)^2,\,(st)^2,\,(sx)^{3m},\,(ty)^4,\, (sxy)^2t,\,txty \rangle
\end{equation}
representing the edge-biregular map $M=M_p$ of type $(k,\ell)=(8,6m)$ that appears in item 4 of Theorem \ref{thm:main-odd}. The corresponding dual map is obtained by interchanging $x$ with $y$ and $s$ with $t$, while the twin map is found by interchanging $x$ with $s$ and $y$ with $t$.
\medskip

\noindent {\bf Proof of correctness of the presentation \eqref{eq:Ht} }
\medskip

It remains to prove that the presentation \eqref{eq:Ht} indeed defines a group of order $k\ell/2 = 24m$. For this we begin with a `universal' group $U$ with relators as in \eqref{eq:Ht} but with $(sx)^{3m}$ omitted:
\begin{equation}\label{eq:Ut}
U=\langle x,y,s,t\, | \, x^2,\,y^2,\,s^2,\,t^2,\,(xy)^2,\,(st)^2,\,(ty)^4,\, (sxy)^2t,\,txty \rangle
\end{equation}
We show that $N=\langle (sx)^3\rangle$ is a normal subgroup of $U$. For this it is sufficient to prove that $N$ is invariant under conjugation by $t$, as from $y=txt$ and automatic invariance of $N$ under conjugation by $s$ and $x$ it then follows that $yNy=N$. From the last two relators in \eqref{eq:Ut} we have $1=(sxy)^2t=(s(ty)^2)^2t$, which, using $(ty)^4=1$, gives $s(ty)^2s=yty$. It follows that conjugation by $s$ fixes $t$ (by the relation $(st)^2=1$) and interchanges $yty$ with $(ty)^2$. On the other hand, as $x=tyt$, one similarly obtains that conjugation by $x$ fixes $(ty)^2$ and interchanges $t$ with $yty$. It follows that the composition of the two conjugations, that is, the mapping $z\mapsto (xs)z(sx)$, induces a $3$-cycle $t\mapsto yty \mapsto (ty)^2\mapsto t$, so that $t$ commutes with $(sx)^3$ and hence preserves $N$. (This is what we saw before but now we needed to make sure that it was established solely from the presentation \eqref{eq:Ut}.)
\smallskip

By the Reidemeister-Schreier theory implemented in MAGMA in the form of its Rewrite command one may check that $N$ is a free group, that is, $N$ is infinite cyclic. Also, by MAGMA one can check that $U/N\cong S_4$, of order $24$. Now, for an arbitrary integer $m\ge 1$ let $N_m= \langle (sx)^{3m}\rangle$ be the cyclic subgroup of $N$ of index $m$. Since $N_m$ is characteristic in $N$ and so normal in $U$, we may use the Third isomorphism theorem to write $U/N\cong (U/N_m)/ (N/N_m)$ and as $U/N_m\cong H$, $N/N_m \cong C_m$ and $U/N\cong S_4$ it follows that $|H|=24m$, as claimed. This proves correctness of the presentation \eqref{eq:Ht}.
\medskip

We conclude by showing that none of the edge-biregular maps given by the group $H_p$ from  \eqref{eq:Ht} is regular. For such a map to be regular there would have to be an automorphism of $H$ of order $2$ interchanging $s$ with $x$ and $t$ with $y$. If this is the case then $H$ would also contain the relator $ysyt$. From the relator $txty$ of $H$ we have $y=txt$, which, when substituted into  $ysyt=1$ and canceling terms, gives $txsx=1$. Combining this with $tysy=1$ yields $xsx=ysy$, or, equivalently, $(sxy)^2=1$. But this in combination with the relator $(sxy)^2t$ of \eqref{eq:Ht} gives $t=1$, a contradiction. This completes  both the analysis related to the fourth item of Theorem \ref{thm:main-odd} as well as our proofs of the main results from section \ref{sec:res}.  \hfill $\Box$

\section{Concluding remarks}

The maps $M_p$ identified in the last part of our proof, those defined by the group $H_p$ presented as in \eqref{eq:Ht}, deserve particular attention. Their structure is best visualized by considering the associated embedded Cayley graph ${\rm Cay}(H_p,X)$ for the group $H_p$ and the generating set $X=\{x,y,s,t\}$. If superimposed onto the map $M_p$, the associated embedding of ${\rm Cay}(H_p,X)$ displays cycles labelled alternately with $y$ and $t$ `around' each vertex and cycles labelled with $x$ and $s$ `around' each face, while the 4-cycles labelled $xyxy$ and $stst$ `surround', respectively, bold and dashed edges. The existence of the relator $txty$ in the presentation \eqref{eq:Ht}, equivalent to the relation $x=tyt$, demonstrates that all the bold edges are loops. Applying this knowledge to the relator $(sxy)^2t$, which (being of odd length) signals non-orientability, one sees that the edges in the unshaded orbit partition into pairs of double-edges, each forming a `central cycle' of a M\"obius strip in the embedding. The Cayley graph ${\rm Cay}(H_p,X)$, part of which is shown in Figure \ref{fig2}, makes this clear. This is also an alternative way of proving non-regularity for this map. Namely, the two edge orbits contain (bold) loops on the one hand, and non-orientable (dashed) 2-cycles on the other, so there will certainly not be an automorphism of the group which interchanges these two orbits of edges.

\begin{figure}
\centering
\begin{tikzpicture}
\clip (-7,6)--(-7,15)--(7,15)--(7,6) -- (-7,6);
 \foreach \x in {88,92,62,58,118,122}
    \draw [line width=2pt] (\x:8) -- (\x:10);
 \foreach \x in {84,54,114}
    \draw [line width=2pt] (\x:8.7) -- (\x+12:8.7);
 \foreach \x in {84,54,114}
    \draw [line width=2pt] (\x:9.3) -- (\x+12:9.3);
 \foreach \x in {84,96,66,54,114,126}
    \draw [line width=1pt] (\x:8.7) -- (\x:9.3);
 \foreach \x in {88,58,118}
    \draw [line width=1pt] (\x:8) -- (\x+4:8) (\x:10) -- (\x+4:10);
 \foreach \x in {84,54,114}
    \draw [dashed, line width=2pt] (\x:8.7) -- (\x+4:8) (\x:9.3) -- (\x+4:10);
 \foreach \x in {96,66,126}
    \draw [dashed, line width=2pt] (\x:8.7) -- (\x-4:8) (\x:9.3) -- (\x-4:10);
  \foreach \x in {88,58,118,148}
    \draw [dash pattern={on 7pt off 3pt}, line width=2pt] (\x:8) -- (\x-26:8);
 \foreach \x in {84,54,114,144}
    \draw [dash pattern={on 7pt off 3pt}, line width=2pt] (\x:8.7) -- (\x-18:8.7);
 \foreach \x in {88,58,28,118}
    \draw [dash pattern={on 7pt off 3pt}, line width=2pt] (\x:10) arc (\x-72:\x+120:3.18);
     \foreach \x in {84,54,24,114}
    \draw [dash pattern={on 7pt off 3pt}, line width=2pt] (\x:9.3) arc (\x-84:\x+110:3.18);

\foreach \x in {88,58,118}
\draw
 (\x:8) node[circle,fill,inner sep=2.5pt]{}
 (\x+4:8) node[circle,fill,inner sep=2.5pt]{}
  (\x:10) node[circle,fill,inner sep=2.5pt]{}
 (\x+4:10) node[circle,fill,inner sep=2.5pt]{};

 \foreach \x in {84,54,114}
\draw
 (\x:8.7) node[circle,fill,inner sep=2.5pt]{}
 (\x+12:8.7) node[circle,fill,inner sep=2.5pt]{}
  (\x:9.3) node[circle,fill,inner sep=2.5pt]{}
 (\x+12:9.3) node[circle,fill,inner sep=2.5pt]{};

   \node[draw, text width=0.1\linewidth,inner sep=2mm,align=center,
      below left] at (current bounding box.north east)
    {Key

$x$
\begin{tikz} \draw [line width=2pt] (0,0)--(1,0); \end{tikz}

$y$
\begin{tikz} \draw [line width=1pt] (0,0)--(1,0); \end{tikz}

$s$
\begin{tikz} \draw [dash pattern={on 7pt off 3pt}, line width=2pt]  (0,0)--(1,0); \end{tikz}

$t$
\begin{tikz} \draw [dashed, line width=2pt]  (0,0)--(1,0) ;\end{tikz}
};

\end{tikzpicture}
\caption{Part of the Cayley graph ${\rm Cay}(H_p,X)$ with generating set $X=\{x,y,s,t\}$}
\label{fig2}
\end{figure}
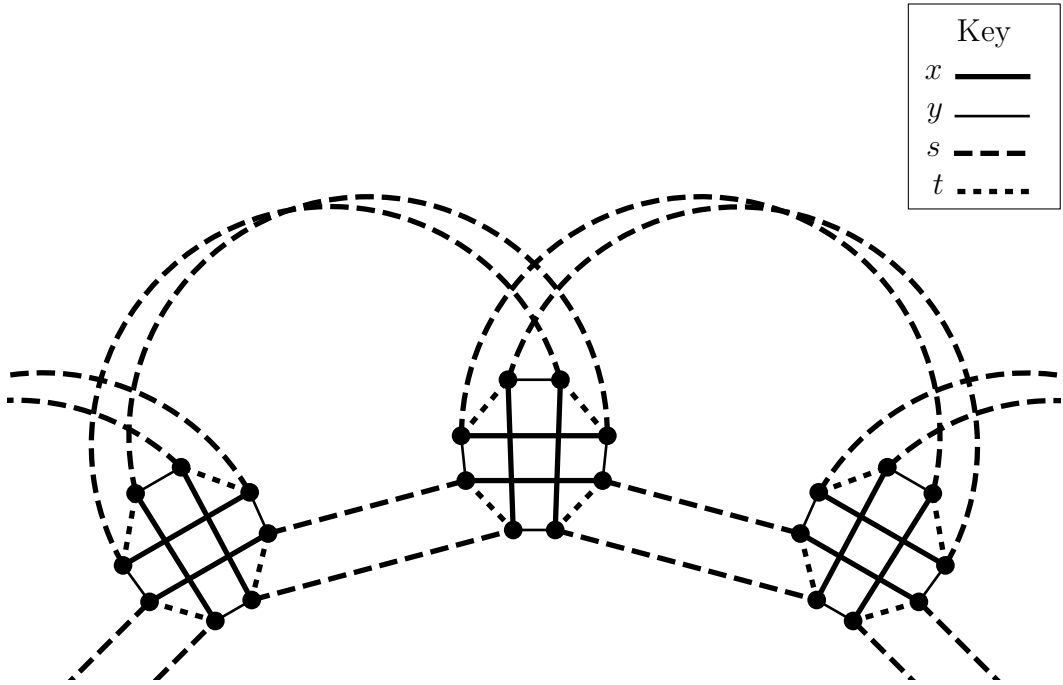

\smallskip
In section \ref{sec:alg} we pointed out that edge-biregular maps of type $(k,\ell)$ arise as smooth quotients of the subgroup of the full triangle group \eqref{eq:Tkl} uniquely determined by inclusion of the generators $R_0$ and $R_2$ while excluding $R_1$. In this context it may be interesting to review results for maps on surfaces of negative prime Euler characteristic arising from smooth quotients of all the (up to) seven possible index-two subgroups of the full triangle groups. Out of these, the subgroup most frequently referred to is $\langle \overline{R}_0, \overline{R}_1,\overline{R}_2\rangle$, the smooth quotients of which form the class of orientably-regular maps. A complete classification of orientably-regular maps of Euler characteristic $-p$ for prime $p$ can be found in \cite{CST}, with earlier results for orientably-regular maps with `large' automorphism groups available in \cite{BJ}. The bi-rotary maps of negative prime Euler characteristic, i.e., those arising from smooth quotients of the subgroup $\langle R_0, \overline{R}_1,\overline{R}_2\rangle$, have been classified in \cite{BCS}. Since the interchange of the involutory generators $R_0$ and $R_2$ induces map duality, analogous classification of maps defined by smooth quotients of the subgroup $\langle \overline{R}_0, \overline{R}_1, R_2\rangle$ of \eqref{eq:Tkl} follows from \cite{BCS} by dualisation.
\smallskip

This way, results in this paper are a further contribution to classification of highly symmetric maps on surfaces as above, this time arising as smooth quotients of the subgroup $\langle R_0,\overline{R}_1, R_2\rangle$ of index two in the full triangle group. In order to have a complete list of highly symmetric maps on surfaces with Euler characteristic $-p$ for prime $p$ that can be obtained as smooth quotients of index-two subgroups of triangle groups it is therefore sufficient to examine, up to duality, only two more subgroups, namely, $\langle \overline{R}_0, R_1, R_2\rangle$ and $\langle \overline{R}_0, R_1,\overline{R}_2 \rangle$. For completeness we recall that the fully regular maps on these surfaces, that is those arising as smooth quotients of the full triangle groups, have been classified in \cite{BNS}.
\medskip
\bigskip

\noindent{\bf Acknowledgements.} \
Both authors attended the Symmetries of Discrete Objects 2020 conference held in Rotorua, New Zealand and would like to thank the organisers Marston Conder and Gabriel Verret for arranging the event and for their generous hospitality. It was at this Symmetries of Discrete Objects workshop that some of the above research was conducted and partial results were presented. The first author is grateful for the travel grants from both the Open University and also from the London Mathematical Society which enabled her to attend the conference.
The second author gratefully acknowledges support from the APVV Research Grants 17-0428 and 19-0308, as well as from the VEGA Research Grants 1/0238/19 and 1/0206/20.
Both authors would like to thank the anonymous referee whose comments made us aware of both errors and omissions in our original manuscript, leading us to make significant improvements to this paper.

\end{document}